\documentclass[journal,10pt,onecolumn,draftclsnofoot,]{IEEEtran}% If IEEEtran.cls has not been installed into the LaTeX system files,

\usepackage{amsmath,amssymb,amsfonts}
\usepackage{amsthm}
\usepackage{hyperref}
\usepackage[labelformat=simple]{subcaption}

\usepackage{mathrsfs}
\usepackage{epsfig} %
\usepackage{graphicx}
\usepackage{graphics}
\usepackage{cite}
\usepackage{psfrag}
\usepackage{epstopdf}
\usepackage{tikz}
\PassOptionsToPackage{usenames,dvipsnames,svgnames}{xcolor} 

\usetikzlibrary{automata, arrows,decorations.pathmorphing,backgrounds,positioning,fit}
\usetikzlibrary{shapes.geometric}
\usetikzlibrary{circuits.ee,circuits.ee.IEC}

\definecolor{mycolor_peach}{RGB}{251,111,66}
\definecolor{mycolor_lightblue}{RGB}{8,180,238}
\definecolor{mycolor_darkblue}{RGB}{1,17,181}
\definecolor{mycolor_teal}{RGB}{18,150,155}
\definecolor{mycolor_green}{RGB}{12,195,82}
\definecolor{mycolor_lightgreen}{RGB}{8,180,238}

\usepackage{tikz-3dplot}
\usepackage{pgfplots}
\usepackage{cleveref}
\newcommand{\bm}[1]{{\boldsymbol{#1}}}

\theoremstyle{definition}

\newtheorem{example}{Example}

\theoremstyle{plain}
\newtheorem{theorem}{Theorem}
\newtheorem{proposition}{Proposition}

\newtheorem{corollary}{Corollary}

\theoremstyle{remark}

\newtheorem{remark}{Remark}

\DeclareMathOperator*{\argmin}{arg\,min}

\begin{document}

\title{Consensus Driven by the Geometric Mean}

\author{Herbert~Mangesius, 
      Dong~Xue,~Sandra~Hirche% <-this % stops a space
\thanks{The authors are with the Chair of Information-Oriented Control, Department of Electrical and Computer Engineering, Technische Universit\"{a}t M\"{u}nchen (TUM), Arcisstrasse 21, D-80209 Munich, Germany. 
e-mail: \texttt{ \{mangesius, dong.xue, hirche\}@tum.de} }
\thanks{The work is partially supported by the German Research Foundation (DFG) within the Priority Program SPP 1914 “Cyber-Physical Networking”, the EU H2020 Innovative Training Network (ITN) "Platform-aware Model-driven Optimization of Cyber-Physical Systems (oCPS)" under grant agreement no. 674875, and the TUM Institute for Advanced Study.}
}

\markboth{H. Mangesius, D. Xue \& S. Hirche}{H. Mangesius, D. Xue \& S. Hirche}

\maketitle

\begin{abstract}
Consensus networks are usually understood as arithmetic mean driven dynamical averaging systems.
In applications, however, network dynamics often describe inherently non-arithmetic and non-linear consensus processes. In this paper, we propose and study three novel consensus protocols driven by geometric mean averaging: a polynomial, an entropic, and a scaling-invariant protocol, where terminology characterizes the particular non-linearity appearing in the respective differential protocol equation. 
We prove exponential convergence to consensus for positive initial conditions. 
For the novel protocols we highlight connections to applied network problems: 
The polynomial consensus system is structured like a system of chemical kinetics on a graph. The entropic consensus system converges to the weighted geometric mean of the initial condition, which is an immediate extension of the (weighted) average consensus problem.
We find that all three protocols generate gradient flows of free energy on the simplex of constant mass distribution vectors albeit in different metrics. 
On this basis, we propose a novel variational characterization of the geometric mean as the solution of a non-linear constrained optimization problem involving free energy as cost function.
We illustrate our findings in numerical simulations. 
%and demonstrate an interesting relationship between consensus states reached by polynomial consensus networks and solutions of an elliptic integral.
\end{abstract}

%\begin{IEEEkeywords}
%consensus systems, non-linear dynamics, geometric mean, free energy, gradient flows
%\end{IEEEkeywords}

\IEEEpeerreviewmaketitle

\section{Introduction}
Under the umbrella of linear consensus theory, results are collected that describe the convergence and stability of a very general class of linear time-varying, arithmetic mean driven network 
dynamics, see, e.g., \cite{Olfati-Saber2007, 
Tsitsiklis1989, Moreau2004}. What suffers from this generality is the specificity that is usually needed to make immediate use of those results in applied network problems - problems, which often 
appear as non-linear and time-invariant dynamics that are inherently driven by non-arithmetic means. A prime example is the class of Kuramoto-type network models \cite{Acebron2005} that can be found in a wide range of 
important applications, e.g., in power grid studies \cite{Doerfler2014, Hendrickx2014}, or in neuroscience \cite{Breakspear2010}. The collective averaging motion is driven by the so-called chordal 
mean, which is an average adapted to the circular geometry of phase angles \cite{Sepulchre2011, SarletteSIAM2009, Scardovi2007}. Important stability results can indeed be based on linear 
consensus theory, see, e.g., the work \cite{Jadbabaie2004}, where the authors prove stability by reverse engineering for this particular case a linear time-varying consensus structure from the non-linear time-invariant original system model.
In this paper the starting point is not an existing non-linear network model, but
a significant type of average, namely the geometric mean, that shall serve as the driving element in a non-linear dynamic averaging network.
In particular, we are interested in designing and studying network protocols which generate geometric mean averaging processes in the same way the arithmetic mean does in linear consensus protocols.
For the novel types of geometric mean driven network dynamics we propose touching points to non-linear problems in chemistry, optimization and analog computation using networked dynamical systems.

The geometric mean plays an important role in various applications.
It is the appropriate tool to evaluate averages on data 
that exhibits power law relationships, as they arise in describing relative, resp., compound growth relations \cite{Spizman2008}. Examples of such relations can be found in financial studies \cite{Zenner2008, Mitchell2004}, they are abundant in biology \cite{Shingleton2010} and chemistry \cite{Connors1990, HaraAutomatica2011}. For instance, in gene expression networks, the geometric mean of degradation and production rates has been found to act as feedback control gain in linearized dynamics \cite{HaraAutomatica2011}.
Geometric mean averaging appears also in the context of algorithm design, in distributed Bayesian consensus filtering and detection 
schemes, see, e.g.,\cite{ChungACC2014} and \cite{QLiu2015}. There, the geometric mean arises from the combination of a given network structure and a Bayesian update rule, leading to a so-called 
logarithmic opinion pooling as natural scheme of combining local probabilities, see \cite{Zidek1986}, and also \cite{NedicArxiv2015,JadbabaieCDC2013,JadbabaieScDir2015} for further reference.

Despite the central role of mean functions and averaging structures for stability studies in network problems,
yet there are only few works on how specific means, and in particular the geometric mean, drive the (non-linear) behavior of such systems.
Consensus-like protocols driven by non-arithmetic
means with geometric mean as particular case are
for instance studied in \cite{Krause2005} in the particular context of opinion dynamics in discrete time.
Works on consensus on non-linear space \cite{SarletteSIAM2009, Sepulchre2011} extend the usual arithmetic averaging in linear consensus to a non-linear configuration space; the 
associated non-arithmetic mean results as a by-product of that choice of geometry.
Extensions to other mathematical structures include the work on consensus on convex metric spaces  \cite{Baras2015}, or on the Wasserstein metric space of probability measures \cite{Bishop2014}.
None of these works puts in the center of consideration a particular type of average, from where continuous-time network dynamics shall arise.

In this work we propose and study three novel non-linear consensus protocols on the basis of elementary considerations on how the arithmetic mean appears in the structure of linear consensus protocols and replacing it by the geometric mean functional relationship.
In particular, we contribute by 
(i) introducing geometric mean driven network protocols that we call polynomial, entropic, and scaling-invariant protocol, 
and
(ii) we prove convergence to consensus under appropriate connectedness conditions.
(iii) We show that the entropic consensus system convergences to the (weighted) geometric mean of the initial state components, which is the geometric mean extension of the usual (arithmetic mean) 
average consensus problem. (iv) We bring the polynomial consensus protocol in relation with reaction rates in chemical kinetics, thus building a bridge between 
consensus theory and biochemistry. (v) We propose a novel variational characterization of the geometric mean in a free energy non-linear constrained minimization problem. (vi) We put the three 
distinct protocols on a common footing by showing that all three protocols describe a particular type of free energy gradient descent flow. 
%(vii) We demonstrate numerically that the consensus value 
%obtained under a polynomial protocol on a normalized balanced graph is lower bounded by the solution of an elliptic integral.

The remainder of this article is organized as follows: 
In section \ref{sec:prelim} we give an overview on mean functions and linear consensus theory.
In section \ref{sec:meandyn} we propose the novel protocols and discuss relationships to arithmetic-mean averaging structures in linear consensus networks.
In section \ref{sec:results} we prove exponential convergence and consensus value results.
%, give illustrative numerical examples, and provide a numerical study on the consensus value for the polynomial consensus system.
In section \ref{sec:energetics}, we put the three novel consensus protocols in a single free energy gradient flow framework and provide a novel variational characterization of the geometric mean. 

% particular gradient flows of free energy, albeit in different (geo)metric settings, and prove a novel 
% variational geometric mean characterization.

\section{Preliminaries \label{sec:prelim}}
In this section, we present basic facts from the fields of mean functions and linear consensus theory.

\subsection{Mean functions}
Consider data points $x_1,x_2,\ldots,x_n$ taking values on the positive real line $\mathbb{R}_{>0}$, and let these elements be collected in the vector $\bm{x}$. An average or mean computed from 
$\bm{x}$ can be obtained as the solution of an unconstrained minimization,
\begin{equation}\label{eq:varmean}
\mathsf{mean}(\bm{x})= \arg \min_{x\in \mathbb{R}_{>0}} \sum_{i=1}^nd(x_i,x)^2,
\end{equation}
where $d(a,b)$ denotes a metric in $\mathbb{R}_{>0}$, i.e., a positive definite and symmetric function, that vanishes iff $a=b$.

If the Euclidean distance $d_\mathsf{E}(a,b):=|a-b|$ is chosen in~\eqref{eq:varmean}, the resulting average is the arithmetic mean,
\begin{equation}\label{eq:amdef}
 \mathsf{am}(\bm{x}):=\frac{1}{n}\sum_{i=1}^nx_i=\argmin_{x\in \mathbb{R}} \sum_{i=1}^n|x_i-x|^2.
\end{equation}

Another important metric in $\mathbb{R}_{>0}$ is the hyperbolic distance~\mbox{$d_\mathsf{H}(a,b):=|\ln a -\ln b|$},
% \begin{equation}\label{eq:dh}
%  d_\mathsf{H}(a,b):=|\ln a -\ln b|,
% \end{equation}
which coincides with the Euclidean metric assessed in logarithmic coordinates.
It is a geodesic distance, and measures the hyperbolic length of the straight line segment joining two points in Cartesian coordinates $(x,a), (x,b)$, $x\in \mathbb{R}_{>0}$, see e.g., 
\cite{Stahl1993} Prop. 4.3.

Its significance arises from the fact that the solution of the minimization problem \eqref{eq:varmean} using the hyperbolic metric $d_\mathsf{H}$ yields the geometric 
mean 
\begin{equation}
\mathsf{gm}(\bm{x}):=\sqrt[n]{x_1x_2\cdots x_n}.
\end{equation}

To see this, observe that
% \begin{align}
% \sum_{i=1}^n|\ln x_i -\ln \mathsf{gm}(\bm{x})|^2&=\sum_{i=1}^n|\ln x_i -\frac{1}{n}\sum_{j=1}^n\ln x_j |^2 \\
% &=\sum_{i=1}^n|\ln x_i -\mathsf{am}(\ln\bm{x})|^2, \label{eq:soslog}
% \end{align}
\begin{equation}
\sum_{i=1}^n|\ln x_i -\ln \mathsf{gm}(\bm{x})|^2=\sum_{i=1}^n|\ln x_i -\mathsf{am}(\ln\bm{x})|^2, \label{eq:soslog}
\end{equation}
which is the least-squares characterization of the arithmetic mean in logarithmic coordinates.

To complete this section we introduce the weighted versions of the arithmetic and geometric means,
\begin{equation}
\mathsf{am}_w(\bm{x}):=\sum_{i=1}^n\omega_ix_i, \ \ \text{and} \ \ \mathsf{gm}_w(\bm{x}):=\prod_{i=1}^nx_i^{\omega_i},
\end{equation}
where for $i=1,2,\ldots,n$, $\omega_i>0$ and $\sum_{i=1}^n\omega_i=1$.

\subsection{Graphs, linear consensus protocols \& the arithmetic mean}
Let $\mathsf{G}=(N,B,w)$ be a weighted digraph (directed graph) with set of nodes $N:=\{1,2,\ldots,n\}$, set of branches
$B:=\{1,2,\ldots,b\}\subseteq N\times N$ having elements ordered pairs $(j,i)$ that indicate that there is a
branch from node $j$ to $i$, and $w:B\to \mathbb{R}_{>0}$ is a weighting function for which we write $w((j,i))=w_{ij}$.
Define the in-neighborhood of a node $i$ as the set of connected nodes $N_i^+:=\{j\in N: (j,i)\in B\}$ and the out-neighborhood $N_i^-:=\{j\in N,(i,j)\in B\}$.
 The (in-)degree of a node $i$ is the value $d_i:=\sum_{j\in N_i^+} w_{ij}$.
 Set $\textbf{D}:=\mathsf{diag}\{d_1,d_2,\ldots,d_n\}$. 
The weighted adjacency matrix $\textbf{W}$ is such that
 $[\textbf{W}]_{ij}=w_{ij}$ for all $(j,i)\in B$; if $(j,i) \not \in B$, then $[\textbf{W}]_{ij}=0$, and $[\textbf{W}]_{ii}=0$, for all 
$i\in N$. A graph is called \emph{balanced} if $\sum_{j=1}^{n}w_{ij}=\sum_{j=1}^{n}w_{ji}$ and it is \emph{symmetric} if $w_{ij}=w_{ji}$, 
$\forall(j,i)\in B$. The Laplacian matrix of a weighted digraph is defined as $\textbf{L}:=\textbf{D}-\textbf{W}$,
and
the normalized Laplacian is $\hat{\textbf{L}}:=\textbf{I}-\hat{\textbf{W}}$, where $\hat{\textbf{W}}=\textbf{D}^{-1}\textbf{W}$ is the 
matrix of normalized branch weights.

A linear consensus system evolving in continuous time is a dynamics on a family of graphs $\{\mathsf{G}(t)\}_{t\geq 0}$ governed
by
\begin{equation}\label{eq:lincons}
\dot{x}_i=\sum_{j\in N_i^+}w_{ij}(t)\left(x_j-x_i\right)\ \ \Leftrightarrow \ \ \dot{\bm{x}}=-\textbf{L}(t)\bm{x},
\end{equation}
where each dynamic branch weight $w_{ij}(\cdot)$ is a measurable non-negative function \cite{Hendrickx2013}. 

% it is assumed that there exists a threshold $\delta >0$ such at for all $t>0$, and for each $(j,i)\in B(t)$, $w_{ij}(t) \geq \delta$, i.e. the finite branch weights are positively bounded away from zero by a minimum threshold value or zero.

The following relationships between the arithmetic mean and consensus system representations and properties are well known in consensus theory: Using \eqref{eq:amdef}, a component-wise LTI consensus dynamics \eqref{eq:lincons} on a normalized weighted digraph can locally be brought to the open-loop control system form
\begin{align}
\dot{x}_i=- x_i +  & u_i(\{ x_j \}_{j\in N_i^+}), \label{eq:amconscontrol}\\
 &u_i= \sum_{j\in N_i^+} \hat{w}_{ij}(t) x_j \triangleq  \mathsf{am}_w(\{x_j\}_{j\in N_i^+}).
\end{align}

Without the requirement of a normalized weighting, a variable time discretization can be chosen such that a local algorithmic update law (e.g. in an explicit Euler scheme)
has the arithmetic mean driven form
\begin{equation}
x_i(t+\mathrm{d}t)=\alpha(\mathrm{d}t)x_i(t) + [1-\alpha(\mathrm{d}t)] \mathrm{am}_w(\{x_j(t)\}_{j\in N_i^+(t)},
\end{equation}
where $0\leq \alpha < 1$, cf., e.g., \cite{Scardovi2007}.

Besides its appearance in the local dynamics at a certain instant in time, the arithmetic average unfolds also as asymptotic global system property: in the class of consensus networks being 
governed by Laplacians $\textbf{L}(t)$ that are irreducible and balanced for all $t\geq 0$, the asymptotically reached uniform agreement value is given by the arithmetic mean 
of the initial condition \cite{Murray2004}.
The problem in which the equilibrium state to be reached is uniform with consensus value $\bar{x}=\mathsf{am}(\bm{x}_0)$ is commonly known as the average consensus problem.

The goal in this work is to study the interplay between 
consensus protocols and the geometric mean.
In that, we first seek to understand the various interaction points between the design of LTI consensus protocols and the arithmetic mean, and then 
leverage these observations to derive and study novel geometric mean driven consensus protocols.

For the sake of focus and ease of understanding, we assume the underlying graph to have a constant, i.e., time-invariant weighting.
In our analysis it shall turn out elementary to transform the non-linear time-invariant network protocols to linear time-varying consensus form, so that
the following convergence result becomes applicable.

\begin{proposition}\label{prop:moreau}[Adopted from \cite{Sepulchre2011} Prop. 1 with Def. 2]
A linear time-varying system evolving according to \eqref{eq:lincons} in $\mathbb{R}^n$ converges globally and exponentially to a consensus point $\bar{x}\bm{1}$, $\bar{x}\in\mathbb{R}$, if the underlying digraph is uniformly connected, i.e., if for all $t>0$, there exists a time horizon $T>0$, such that the graph $(N,\tilde{B}(t),\tilde{w}(t))$ defined by
\begin{equation}
\tilde{w}_{ij}(t):= \left\{ \begin{tabular}{l l}
$\int_t^{t+T} w_{ij}(\tau)\mathrm{d}\tau$ & if  \; $\int_t^{t+T} w_{ij}(\tau)\mathrm{d}\tau \geq \delta >0$ \\
$0$ & if \; $\int_t^{t+T} w_{ij}(\tau)\mathrm{d}\tau < \delta$
\end{tabular}\right.
\end{equation}
$w_{ij}(\tau)$ a branch weight at time $\tau$, $(j,i) \in B$ if and only if $\tilde{w}_{ij}(t)\not = 0$, contains a node from which there is a path to every other node.
%For any initial condition in $\mathbb{R}^n$, a linear time-varying system \eqref{eq:lincons}
%converges with exponential speed to a uniform equilibrium point $\bar{x}\bm{1}$, $\bar{x}\in\mathbb{R}$, 
%if the sequence of graphs over time is persistently connected, i.e.,
%if the graph $(N,\tilde{B})$, where
%$\tilde{B}:=\left\{(j,i)\in N\times N: \int_0^\infty w_{ij}(t)\mathrm{d}t=+\infty\right\}$,
%is strongly connected.
\end{proposition}
 Uniform connectivity certainly holds if at 
each time instant the graph $\mathsf{G}(t)$ is strongly connected and $w_{ij}(t)\geq \delta>0$, i.e., if the graph contains a directed path from every node to every other node and the finite branch 
weights are positively bounded away from zero for all time.

\section{Geometric mean driven Network protocols\label{sec:meandyn}}
In this section we propose and motivate three novel geometric mean driven network protocols.

\subsection{Polynomial protocol}
The polynomial protocol we consider is a dynamics on a graph where at each node $i \in N$, the differential update rule has the form 
\begin{equation}\label{eq:polyprot}
\dot{x}_i=-\prod_{j\in N_i^-} x^{w_{ji}}_i +\prod_{j\in N_i^+} x^{w_{ij}}_j.
\end{equation}
For the polynomial protocol we assume a balanced graph weighting, i.e., $\sum_{j\in N_i^-}w_{ji}=d_i$.
With that, the protocol \eqref{eq:polyprot} can be written as
\begin{equation}
\dot{x}_i=-x^{d_i}_i +\prod_{j\in N_i^+} x^{w_{ij}}_j.
\end{equation}
Comparing this form with a linear consensus protocol, which can be stated as
\begin{equation}
\dot{x}_i=-d_ix_i +\sum_{j\in N_i^+} w_{ij}x_j, \ \ \ i \in N
\end{equation}
we observe an equivalence resulting upon replacing the operation of summation and multiplication with the similar\footnote{These operations are similar in the sense that the addition of logarithmic variables turns the variables into products, and products result into exponentiation.} operations multiplication and exponentiation. 

Alternatively, referring to the open-loop control representation \eqref{eq:amconscontrol}, where weightings are normalized, replacement of the weighted arithmetic mean by the geometric average leads to the protocol
\begin{equation}
\dot{x}_i=-x_i+\mathsf{gm}_w(\{x_j\}_{j\in N_i^+})=-x_i + \prod_{j\in N_i^+}x_j^{\hat{w}_{ij}},
\end{equation}
from where \eqref{eq:polyprot} results again under the assumption of 
having a balanced weighting, that is, $\sum_{j\in N_i^-}\hat{w}_{ji}=1$.

In its general form \eqref{eq:polyprot}, the polynomial protocol has the structure of a rate equation as it occurs for instance in reaction networks and chemical kinetics \cite{Connors1990}.
We define
\begin{equation}\label{eq:rip}
r_i^+:=\prod_{j\in N_i^+}x_j^{w_{ij}},
\end{equation} 
the non-linear rate at which some quantity ``$x$'' flows from in-connected nodes $j$ to node $i$, and
\begin{equation}\label{eq:rim}
r_i^-:=\prod_{j\in N_i^-}x_i^{w_{ji}},
\end{equation}
the rate at which $x$ flows along links $(i,j)\in B$ from node $i$ to the out-directed nodes $j$. The local rate of change $\dot{x}_i$ balances in- and out-flows on a graph, as $\dot{x}_i= 
r_i^+(\bm{x})-r_i^-(\bm{x})$.
% \begin{equation}\label{eq:ratediffeq}
% \dot{x}_i= r_i^+(\bm{x})-r_i^-(\bm{x}).
% \end{equation}
The relation to the (weighted) geometric mean and the similarity to chemical kinetics in reaction networks are described in the following example.

\begin{example}[Chemical kinetics]
In mass action chemical reaction networks, the net rate equation for a concentration of one component $i$ in one reaction involving $n$ 
substances indexed in $N$ is split into a difference of a forward and a backward reaction rate, each having the form
\begin{equation}\label{eq:lnkinetic}
r_i^\pm=k^\pm \prod_{j=1}^n x_j^{s_{j}^\pm}=e^{ \sum_j s_j^\pm\ln x_j +\ln k^\pm }
\end{equation}
where $\pm$ stands either for the forward or backward rate,
   and $k^\pm>0$ is the associated forward/backward reaction constant.   
The weights $s_{j}^\pm>0$ are stoichiometric coefficients.
The representation \eqref{eq:lnkinetic} has been instrumental in the recent studies
\cite{vdSchaft2013,vdSchaftIFAC2013,Yong2012} that shed light on a systems theoretic structure of chemical reaction networks: Introducing the density vector $\bm{\rho}$, with $\rho_i=\frac{x_i}{\bar{x}_i}, i\in N$,
under a detailed balance assumption on the equilibrium concentrations $\bar{\bm{x}}$, it can be shown that 
\begin{equation}
\sum_j s_j^\pm\ln x_j +\ln k^\pm =\sum_j s_j^\pm\ln \rho_j.
\end{equation}
Observe that
\begin{equation}
e^{ \sum_j s_j^\pm\ln \rho_j }=e^{ \ln \prod_j  \rho_j^{s_j^\pm }} = \mathsf{gm}_w(\bm{\rho}),
\end{equation}
i.e., the (forward or backward) reaction rate has the functional structure of a (non-normalized) weighted geometric mean.
\end{example}

\subsection{Entropic protocol}
%Next we introduce the so-called entropic network protocol, which (as we shall prove in section \ref{sec:results}) finds application in computing the (weighted) geometric mean of a set of 
%positive real 
%numbers in an analog and distributed manner, i.e. using a continuous-time dynamical network system. The entropic protocol is the geometric mean version of a linear consensus protocol, when seen as 
%distributed (Euclidean) gradient descent scheme of a metric interaction potential (a generalization of the well-known Laplace potential).

The entropic protocol is governed by a vector field
that is represented by a set of negative (weighted) divergences between local states $x_i$ and connected nodes' states $x_j$, such that
\begin{equation}\label{eq:relentflow}
\dot{x}_i= - \sum_{j\in N_i^+}w_{ij} x_i\ln \frac{x_i}{x_j}.
\end{equation}
The term ``entropic'' refers to the fact that a local vector field \eqref{eq:relentflow} is an entropic quantity. More precisely, it has the structure of a negative relative entropy / information 
divergence between the local state $x_i$ and the adjacent states $x_j$, $j\in N_i^+$.
Relative entropy as divergence from a positive vector $\bm{x}$ to another positive vector $\bm{y}$, both such that their $1$-norms equal one, (i.e., these are
probability mass vectors), is defined as
\begin{equation}\label{eq:relent}
D_{\mathsf{ent}}(\bm{x}||\bm{y}):=\sum_i f_R(x_i|y_i), \ \ \text{where} \ \
f_R(a|b):=a\ln\frac{a}{b},
\end{equation}
see for instance ~\cite{Cover1991}.

The entropic protocol can be formulated as the geometric mean version of the linear consensus protocol using a coordinate transformation, with coordinate transform taken as the scalar function that leads to a least squares variational characterization of the considered mean; for the geometric mean this is the logarithm, while for the arithmetic mean no coordinate transformation is required, see \eqref{eq:amdef} with \eqref{eq:soslog}.

Writing the consensus protocol in logarithmic coordinates leads for each $i \in N$ to the ODE
\begin{align}
\frac{\mathrm{d}}{\mathrm{d}t} \ln x_i= \frac{1}{x_i}\dot{x}_i &=\sum_{j\in N_i^+} w_{ij}(\ln x_j-\ln x_i) \\ \Leftrightarrow \ \ \dot{x}_i &= x_i\sum_{j\in N_i^+} w_{ij}(\ln x_j-\ln x_i),
\end{align}
which is the entropic protocol \eqref{eq:relentflow}, as
\begin{equation}
x_i\sum_{j\in N_i^+} w_{ij}(\ln x_j-\ln x_i)=-\sum_{j\in N_i^+} w_{ij}f_R(x_i,x_j).
\end{equation}

%On a symmetrically weighted graph, i.e., $\textbf{W}$ is selfadjoint, the entropic protocol can be formulated as a geometric mean version of the linear consensus protocol using
%an interaction potential in a gradient approach.
%We define
%the metric interaction potential, $\Psi(\bm{x}):=\sum_{(j,i)\in B}w_{ij} d(x_i,x_j)/2$.
%% \begin{equation}\label{eq:intpot}
%% \Psi(\bm{x}):=\frac{1}{2}\sum_{(j,i)\in B}w_{ij} d(x_i,x_j).
%% \end{equation}
%This is a generalization of the usual Laplace potential, which yields
%a symmetric LTI consensus system as the gradient descent ODE
%\begin{equation}
%\dot{\bm{x}}=-\nabla\Psi_\mathsf{E}(\bm{x}), \ \ \Psi_\mathsf{E}:=\frac{1}{2}\bm{x}^\top\textbf{L}\bm{x}=\frac{1}{2}\sum_{(j,i)\in B} w_{ij}d_\mathsf{E}^2(x_i,x_j).
%\end{equation}

%While the linear consensus protocol being associated with the Euclidean metric through the interaction potential, and the Euclidean distance being associated to the arithmetic mean through the variational characterization \eqref{eq:varmean}, 
%the geometric mean is associated with the hyperbolic distance so that the corresponding hyperbolic interaction potential reads
%\begin{align}
%\Psi_\mathsf{H}(\bm{x}):=&\frac{1}{2} \sum_{(j,i)\in B} w_{ij} d_\mathsf{H}(x_j,x_i)^2 \\
%=&\frac{1}{2}\sum_{(j,i)\in B}w_{ij}(\ln x_j +\ln \frac{1}{x_i} )^2.
%\end{align}
%Following again a gradient approach we then obtain
%\begin{equation}
%\dot{\bm{x}}_i=-\nabla \Psi_\mathsf{H}(\bm{x}) \ \ \Leftrightarrow \ \ \dot{x}_i=\sum_{j\in N_i^+} w_{ij}(\ln x_j-\ln x_i)x_i,
%\end{equation}

As we shall show, the significance of the entropic protocol arises from the situation that the asymptotically reached consensus value is given by the geometric mean of the initial condition. Hence, 
this protocol provides an analog distributed computation routine to solve the minimization \eqref{eq:varmean} associated to the geometric mean. 
%In section \ref{sec:energetics}, we provide a novel alternative gradient descent and minimization view involving free energy (as generalization of relative entropy).

%\begin{remark}
%Beyond the well-known LTI consensus case, the gradient approach applies also to the well-known Kuramoto-like phase-averaging dynamics, cf, e.g.,  \cite{Sepulchre2011,SarletteSIAM2009}.
%The appropriate metric in the dissipation potential is the so-called chordal distance $d_C(a,b):=\sqrt{(1-\cos(a-b))}$.
%\end{remark}
%In the following example we show the importance of the gradient approach in symmetrically weighted systems from where also the well-known Kuramoto-like phase-averaging dynamics can be obtained using the appropriate metric related to the chordal mean, cf., e.g.,.

%\begin{example}[Phase averaging]
%To derive the consensus equation on the circle/torus, the gradient approach requires consideration of the metric dissipation potential
%\begin{equation}
%\Psi_\mathsf{C}(\bm{x}):=\frac{1}{2}\sum_{(j,i)\in B}w_{ij} d^2_\mathsf{C}(x_i,x_j),
%\end{equation}
%where $d_C(a,b):=\sqrt{(1-\cos(a-b))}$. This is the ``chordal distance'' and it represents the Euclidean distance between two points on a circle characterized by phase angles $a,b$.
%Using this interaction potential, the non-linear phase averaging protocol is obtained as
%\begin{equation}
%\dot{\bm{x}}=-\nabla \Psi_\mathsf{C} \Leftrightarrow \dot{x}=\sum_{j\in N^+} w_{ij}\sin(x_j-x_i),
%\end{equation}
%which is the coupling dynamics in the famous Kuramoto model of (phase)-coupled oscillator networks.
%\end{example}

\subsection{Scaling-invariant protocol}
In this section, we introduce the scaling-invariant protocol as an instance of a novel class of network dynamics driven by pairwise metric interactions. 

The scaling-invariant protocol has the form of a LTI consensus system however following log-linear updates; it is given by the component ODE
 \begin{equation}\label{eq:scaleinvprot}
 \dot{x}_i=\sum_{j\in N_i^+}w_{ij}(\ln x_j-\ln x_i), \ \ i \in N.
 \end{equation}
This is an instance of the more general type of mean-driven network protocols given by the class
\begin{equation}\label{eq:metriccons}
\dot{x}_i=\sum_{j\in N_i^+}w_{ij} \mathsf{sgn}(x_j-x_i) d(x_j,x_i),
\end{equation}
where the metric to be chosen is the hyperbolic metric $d_\mathsf{H}$ associated to the variational characterization of the geometric mean, see section \ref{sec:prelim}.

The general mean-driven equation \eqref{eq:metriccons} can be motivated from a system thermodynamic viewpoint; in \cite{Haddad2008} a network protocol is proposed with pairwise interactions of the form $f(x_i,x_j)$, where $f$ is 
locally Lipschitz continuous and assumed to satisfy the condition $(x_i-x_j)f(x_i,x_j)\leq 0$,  $f(x_i,x_j)=0$ if $x_i=x_j$. According to the authors this assumption implies that some sort of energy or information 
flows from higher to lower levels thus this condition is reminiscent of a ``second law''-like inequality in thermodynamics.

We observe that this negativity hypothesis is naturally fulfilled by a metric interaction form as in \eqref{eq:metriccons}:
 for any two states the sign of the terms $(x_i-x_j)$ and $f(x_i,x_j)$ must differ. Hence, for two arguments $x_i,x_j$, $f$ has the sign $\mathsf{sign}(x_j-x_i)$. Therefore, we get the structure $f= \mathsf{sign}(x_j-x_i)f_{\mathrm{res}}$ with residual part required to be positive definite. The choice $f_{\mathrm{res}}=d$ follows naturally.

\begin{example} When the metric chosen in the local ODEs is the Euclidean distance, we recover the linear consensus protocol. Olfati-Saber and Murray's non-linear consensus protocol \cite{MurrayACC2003},
\begin{equation}\label{eq:Murray}
\dot{x}_i=\sum_{j\in N_i^+}w_{ij} \phi(x_j-x_i),
\end{equation}
where $\phi$ is a continuous, increasing function that satisfies $\phi(0)=0$, is 
a subclass of a network dynamics  \eqref{eq:metriccons}. The non-linear interaction in phase averaging, $\phi(\cdot)=\sin(\cdot)$ on the open interval $]-\pi/2, \pi/2[$ is a famous example.
%An instance of admissible interaction types that are not covered by Olfati-Saber and Murray's class are given by $h(x_j)-h(x_i)$, where $h$ is an increasing function. This leads to models for non-linear diffusions in porous media \cite{ErbarMaas2014}. 
\end{example}

%This simple protocol is interesting as it lies at the heart of all three protocols: As we shall see in section \ref{sec:energetics}, the proposed three network systems are all versions of a gradient flow of free energy on appropriate $n-1$ dimensional subsets of $\mathbb{R}^n$.

\section{Convergence to consensus \label{sec:results}}
In this section we show global exponential convergence to a consensus configuration 
of the three network protocols driven by the geometric mean, which in summary are given by
\begin{align}
\dot{x}_i&=-\prod_{j\in N_i^-}x_i^{w_{ji}}+\prod_{j\in N_i^+}x_i^{w_{ij}}, \label{eq:1} \\
\dot{x}_i&= -\sum_{j\in N_i^+}w_{ij} x_i\ln\frac{x_i}{x_j}, \ \ \text{and} \label{eq:2} \\
\dot{x}_i&=\sum_{j\in N_i^+}w_{ij}(\ln x_j-\ln x_i).\label{eq:3}
\end{align}
For protocols \eqref{eq:2} and \eqref{eq:3} we characterize the reached consensus value analytically\footnote{An upper and lower bound of the consensus value obtained via protocol \eqref{eq:1} is demonstrated  in the appendix by means of numerical simulations.}.

\subsection{Consensus points and exponential convergence}
To study stability of fixed-points we shall make use of the logarithmic mean, its properties and the mean value theorem:
The logarithmic mean of two positive real numbers $a,b$ is defined as
\begin{equation}
\mathsf{lgm}(a,b):=\frac{a-b}{\ln a -\ln b}.
\end{equation}

The logarithmic mean is positive and symmetric in both arguments, i.e., $\mathsf{lgm}(a,b)=\mathsf{lgm}(b,a)$.
The mean value theorem states that for a continuously differentiable function $f: [a,b]\subseteq \mathbb{R}\to\mathbb{R}$, there exists a $\xi \in [a,b]$ such that
\begin{equation}
\nabla f(\xi)=\frac{f(b)-f(a)}{b-a}.
\end{equation}
With $f=\ln$, we get $\mathsf{lgm}(a,b)=\xi$, where $0<a\leq \xi \leq b$.

The logarithmic mean and its inverse take positive and finite values for positive and finite arguments. 
For approaching positive real arguments we further have
\begin{equation}
\lim_{b\to a} \frac{\ln b -\ln a}{b-a}=\lim_{\epsilon\to 0^+} \frac{\ln (a+\epsilon) -\ln a}{\epsilon}\triangleq
\left.\nabla \ln \xi \right\vert_{\xi=a}=\frac{1}{a},
\end{equation}
so that $\lim_{b\to a}\mathsf{lgm}(a,b)=a>0$.

\begin{theorem}[Convergence to consensus]\label{thm:converg}
Consider  network protocols \eqref{eq:1}-\eqref{eq:3} with initial conditions restricted to $\mathbb{R}^n_{>0}$.
If the underlying digraph is strongly connected,
then protocols \eqref{eq:2} and \eqref{eq:3} converge exponentially fast to a consensus configuration. If in addition the weighting is balanced, then protocol \eqref{eq:1}
converges exponentially fast to a consensus state. In all three cases the equilibrium
$\bar{x}\bm{1}$ has agreement value $\min_{i\in N} x_i(0)< \bar{x} <\max_{i\in N} x_i(0) $.
\end{theorem}

\begin{proof}
We start with \eqref{eq:1} from where the two other cases shall follow.
As $r_i^+$ and $r_i^-$, as defined in \eqref{eq:rip} and \eqref{eq:rim}, are positive, we can expand the protocol \eqref{eq:1} with the logarithm of these rates, so that,
\begin{align}
\dot{x}_i& = r_i^+-r_i^-=\frac{r_i^+-r_i^-}{\ln r_i^+ - \ln r_i^-} \left(\ln r_i^+ -\ln r_i^-\right) \\
&= \mathsf{lgm}(r_i^+,r_i^-) \left( \sum_{j\in N_i^+}w_{ij}\ln x_j - \sum_{j\in N_i^-} w_{ji} \ln x_i\right) \label{eq:dynlg}\\
&= \mathsf{lgm}(r_i^+,r_i^-)  \sum_{j\in N_i^+}w_{ij}\left(\ln x_j-\ln x_i \right)  \label{eq:dynlg2}.
\end{align}
In going from \eqref{eq:dynlg} to \eqref{eq:dynlg2} we made use of balancedness of 
the weighting, so that $\sum_{j\in N_i^-} w_{ji}=\sum_{j\in N_i^+} w_{ij}$.
Expanding 
the pairwise interactions by local pairwise state differences yields
\begin{equation}\label{eq:proof1}
\dot{x}_i=\mathsf{lgm}(r_i^+,r_i^-) \sum_{j\in N_i^+}w_{ij} \frac{\ln x_j-\ln x_i}{x_j-x_i}  \left( x_j- x_i\right), \ \ i\in N.
\end{equation}
Define the matrix $\textbf{L}_X(\bm{x}(t))$,
\begin{equation}\label{eq:com_laplacian}
[\textbf{L}_X]_{ij}:=
\left\{
\begin{array}{cc}
-w_{ij} \mathsf{lgm}^{-1}(x_j,x_i), & \text{if} \ \ j\not =i,\\
\sum_{j\in N_i^+}   w_{ij} \mathsf{lgm}^{-1}(x_j,x_i), & j=i, i\in N, 
\end{array}\right.
\end{equation}
and
$\textbf{R}:=\mathsf{diag}\{\mathsf{lgm}(r_1^+,r_1^-),\mathsf{lgm}(r_2^+,r_2^-),\ldots, \mathsf{lgm}(r_n^+,r_n^-)\}$. 

Then, we get the vector-matrix representation for the polynomial ODE system,
\begin{equation}\label{eq:linform3}
\dot{\bm{x}}=-\textbf{R}(\bm{x})\textbf{L}_X(\bm{x})\bm{x}.
\end{equation}

Next we show that for positive initial conditions the flow generated by the ODE system \eqref{eq:linform3} is well defined for all time:
For $\bm{x}(0)\in \mathbb{R}_{>0}^n$, the matrix $\textbf{L}_X(\bm{x}(0))$ by definition is a Laplacian matrix with finite, non-negative and real off-diagonal elements, as the branch weights are non-negative and the logarithmic mean of positive, real and finite arguments is positive, real and finite. This follows from the mean value theorem: For $x_i,x_j \in \mathbb{R}_{>0}$, $\mathsf{lgm}^{-1}(x_i,x_j)=\frac{1}{\xi}>0$, as $\xi$ is a value within the interval spanned by the positive real numbers $x_i$ and $x_j$. Hence for positive initial condition one can always find a threshold $\delta_X$, such that $\mathsf{lgm}^{-1}(x_i(0),x_j(0))\geq\delta_X>0$.
\noindent
The diagonal matrix $\textbf{R}(\bm{x}(0))$ is positive definite, as for positive initial conditions $r_i^+$ and $r_i^-$ are positive, so that $\mathsf{lgm}(r_i^+,r_i^-)>0$ as well, with value in between the two rates, again by the mean value theorem. Hence, with positive initial condition, one can always find a lower bound $\delta_R>0$, such that $\mathsf{lgm}(r_i^+,r_i^-)\vert_{t=0}\geq \delta_R>0$.
\noindent
Therefore, the matrix $\textbf{R}(\bm{x}(0))\textbf{L}_X(\bm{x}(0))$ is a Laplacian matrix characterizing a ``virtual'' graph with non-negative finite entries, and non-trivial ``virtual'' branch weights that are bounded away from zero by a threshold value $\delta$ such that $ \delta \geq \min_{(j,i)\in 
B}\{w_{ij}\}\cdot \delta_R\cdot\delta_X >0$. Hence, at $t=0$ the polynomial ODE system defines a consensus network.
\noindent
By definition, the flow map of a consensus system is a stochastic matrix, which is a positive
monotone map that leaves $\mathbb{R}_{> 0}^n$ invariant, cf., e.g., \cite{Sepulchre2010} for this monotonicity fact in consensus theory.
Thus, trajectories starting in $\mathbb{R}_{> 0}^n$ will remain in this set, so that $[\textbf{R}\textbf{L}_X](\bm{x}(t))$ is well-defined for all $t\geq 0$, and it characterizes a linear time-varying consensus network, where the variability of ```virtual'' branch weights is endogenously 
determined as function of state trajectories, which in turn are parameterized by time as free parameter.

As the graph $\mathsf{G}$ on which the protocols run is strongly connected by hypothesis, the ``virtual'' graph associated to the dynamic Laplacian $[\textbf{R}\textbf{L}_X](\cdot)$ is 
uniformly connected at each time instant for all $\bm{x} \in \mathbb{R}_{>0}^n$.
Therefore, the polynomial network protocol converges globally and exponentially to a consensus configuration $\bar{\bm{x}}\in \mathsf{span}\{\bm{1}\}$, according to Proposition \ref{prop:moreau}.

Now we consider protocol \eqref{eq:2} and relax the constraint of balanced to arbitrary weighting of the strongly connected graph.
Define the matrix $\textbf{X}(\bm{x}):=\mathsf{diag}\{x_1,x_2,\ldots,x_n\}$. 
Protocol \eqref{eq:2} can be written as
\begin{align}
\dot{x}_i&=x_i \sum_{j\in N_i^+}w_{ij}\mathsf{lgm}^{-1}(x_i,x_j)(x_j-x_i) \\
  \Leftrightarrow \ \ \dot{\bm{x}}&=-\textbf{X}(\bm{x})\textbf{L}_X(\bm{x})\bm{x}. \label{eq:XLx}
\end{align}
The matrix $\textbf{X}\textbf{L}_X$ is a Laplacian matrix for all parameterizations, by the same arguments as before, so that also the 
entropic protocol \eqref{eq:2} converges to a consensus configuration with exponential speed on the positive orthant.

The last protocol \eqref{eq:3} can be written as
\begin{equation}\label{eq:3proof}
\dot{\bm{x}}=-\textbf{L}_X(\bm{x}(t))\bm{x},
\end{equation}
which again is a linear time-varying consensus system with endogenously determined variability of the weighting. Hence, the system converges to consensus with exponential speed on the positive orthant, as well.

Let us turn to the last statement regarding the exponentially fast reached consensus value.
All three non-linear protocols can be brought to a linear consensus form on a dynamically weighted but strongly connected ``virtual'' graph. By standard linear consensus theory, the 
function $\max_{i\in N}x_i-\min_{i\in N}x_i$ is a (strict) Lyapunov function \cite{Moreau2005}.
Hence, the maximal state value is decreasing and the minimal state value is increasing, so that the consensus value must lie in between the initial maximum and minimum state values.
\end{proof}

This proof technique is of interest in its own right: we make use of the Laplacian structure arising from (algebraic) interconnections on a graph in shifting non-linearity associated to nodes to a 
non-linearity in pairwise interactions across branches, leading to a ``virtual'' dynamic graph 
on which the non-linear time-invariant network dynamics appears as linear time-varying consensus system.

%eventually allowing to reverse engineer a linear time-varying consensus structure with endogenous time-variance in branch weights. 

\begin{remark}[Time-varying graphs and uniform connectedness]
We note that the transformations in the proof of convergence do not rely on time-invariant weightings. This suggests that convergence to consensus should take place also under the weaker assumption of uniform graph connectivity. 
\end{remark}

In the following result we analytically characterize the consensus value of
 the entropic consensus network.

\begin{theorem}[Weighted geometric mean consensus]\label{thm:wgm}
Consider a weighted digraph
that is strongly connected, with left eigenvector of the associated Laplacian $\textbf{L}$, $\bm{q}\in \mathbb{R}^n_{>0}$, such that $\bm{q}^\top\textbf{L}=\bm{0}$.
Then, the consensus dynamics \eqref{eq:2} starting at any $\bm{x}(0)\in \mathbb{R}^n_{>0}$ asymptotically reaches a fixed point $\bar{x}\bm{1}$,
with 
\begin{equation}
\bar{x}=\mathsf{gm}_w(\bm{x}(0))=\prod_{i=1}^nx_i^{\hat{q}_i}(0),
\end{equation}
where $\hat{\bm{q}}:=\bm{q}/|\bm{q}|_1$, is the Perron vector of $\textbf{L}$.
\end{theorem}
\begin{proof}
The proof concerning convergence to consensus is analogous to the proof of Theorem \ref{thm:converg}, where the non-linear time-invariant 
system of equations is transformed to a linear time-varying consensus form, with endogenously determined variability of the branch weights.
In particular, let us start from the linear representation of the protocol in vector matrix form \eqref{eq:XLx}, which can be re-written as
\begin{equation}
\textbf{X}^{-1}(\bm{x})\dot{\bm{x}}=\textbf{L}_X(\bm{x})\bm{x} \Leftrightarrow \frac{\mathrm{d}}{\mathrm{d}t}\ln \bm{x} = \textbf{L} \ln 
\bm{x},
\end{equation}
as $\frac{\mathrm{d}}{\mathrm{d}t}\ln x(t)= \frac{1}{x}\dot{x}$ and the Laplacian structure allows to shift the non-linearity from inverted logarithmic mean components in the weightings to logarithmic coordinates at nodes such that
\begin{equation}\label{eq::coordinate_trans}
\textbf{L}_X\bm{x}=\textbf{L}\ln\bm{x}. 
\end{equation}
Note that the inverse $\textbf{X}^{-1}$ exists, as it is a diagonal matrix having positive real diagonal elements.

Next we prove that the weighted geometric mean is the consensus value.
By hypothesis, $\bm{q}$ is in the left kernel of $\textbf{L}$, so that $\bm{q}^\top \textbf{L}\ln \bm{x}=0$.  Equivalently, 
\begin{equation}\label{eq:invprop}
\frac{\mathrm{d}}{\mathrm{d}t} [\bm{q}^\top\ln\bm{x}(t)]=0 \Rightarrow \bm{q}^\top\ln\bm{x}(0)=
\sum_{i=1}^n q_i \ln x_i(t)=\text{const}.
\end{equation}
Using the fact that for $t\to\infty$ a uniform state is reached, together with basic arithmetics for the logarithm,
the invariance property \eqref{eq:invprop} implies that
\begin{align}
\sum_{i=1}^n q_i \ln \bar{x}&=\sum_{i=1}^n q_i \ln x_i(0) \\
\Leftrightarrow \ \ \ln \bar{x}&=
\frac{1}{\sum_{i=1}^n q_i }\sum_{i=1}^n q_i \ln x_i(0)=\ln \prod_{i=1}^n x_i(0)^{\hat{q}_i}.
\end{align}
Solving for the consensus value yields,
\begin{equation}
\bar{x}=\exp\left( \ln \prod_{i=1}^n x_i(0)^{\hat{q}_i} \right)  \triangleq \mathsf{gm}_w(\bm{x}(0)),
\end{equation}
which completes the proof.
\end{proof}

\begin{corollary}\label{cor:wam}
Consider the scaling invariant protocol \eqref{eq:3} in the setting described in Theorem \ref{thm:wgm}.
The asymptotically reached consensus value is the weighted arithmetic mean of the initial 
condition with weights given by the components of the Perron vector, i.e., $\bar{x}=\sum_{i=1}^n \hat{q}_i x_i(0)$.
\end{corollary}
\begin{proof}
The proof follows from noting that $\sum_{i=1}^n \hat{q}_i x_i(t)$ remains invariant along the dynamics.
\end{proof}

Regarding the asymptotically reached agreement value of the polynomial consensus protocol the maximum and minimum initial state values provide upper and lower bounds by standard linear consensus theory.
So far we could not analytically derive tighter results. However, comprehensive numerical simulations, see appendix, suggest that the consensus value is upper bounded by the arithmetic mean of the initial condition and lower bounded by the arithmetic geometric mean of the arithmetic mean and the geometric mean of the initial state. This value is related to the solution of an elliptic integral. The proof of this conjecture is subject to future work. For further information we refer the interested reader to the appendix and references stated therein.

%Before we proceed with a characterization 
%of bounds for the consensus value of the polynomial network dynamics, we first illustrate the results derived so far.

\begin{figure*}
\begin{subfigure}[b]{0.36\textwidth}
 \begin{subfigure}[b]{\textwidth}
\resizebox{1\textwidth}{!}{\begin{tikzpicture}[>=stealth',shorten >=1pt, node distance=3cm, on grid,initial/.style={}]
\tikzstyle{every state}=[fill=mycolor_lightblue]
  \node[state]          (3)                        {$\mathbf{3}$};
  \node[state]          (2) [above left =of 3]    {$\mathbf{2}$};
  \node[state]          (1) [left =of 2]    {$\mathbf{1}$};
  \node[state]          (4) [above right =of 3]    {$\mathbf{4}$};
 \node[state]          (5) [right=of 4]    {$\mathbf{5}$};
 %\node[state]          (T) [below=of 3] {text};
\tikzset{mystyle/.style={->,bend right=16, draw=mycolor_green, double=mycolor_green}}
 \path (2)     edge [mystyle]    node [above]  {$2$} (1)
       (2)     edge [mystyle]    node [sloped, above]  {$2$} (3)
       (4)     edge [mystyle]    node  [above] {$1$} (3)
       (5)     edge [mystyle]    node  [above] {$1$} (4);
\tikzset{mystyle/.style={->, bend right=16, draw=mycolor_green, double=mycolor_green}}   
\path (3)     edge [mystyle]    node  [sloped, above] {$3$} (4)
      (1)     edge [mystyle]    node  [above] {$1$} (2)
      (1)     edge [mystyle]    node  [sloped, below] {$2$} (3)    
      (3)     edge [mystyle]    node  [sloped, above] {$3$} (2)     
      (4)     edge [mystyle]    node  [above] {$4$} (5); 
\tikzset{mystyle/.style={->, bend left=16,draw=mycolor_green, double=mycolor_green}}       
\path (5)     edge [mystyle]    node  [sloped, below] {$3$} (3);
\tikzset{mystyle/.style={->, draw=mycolor_green, double=mycolor_green}}
\path (3)     edge [mystyle]    node [sloped, above]  {$1$} (1)
      (3)     edge [mystyle]    node  [above ] {$1$} (5);
 \end{tikzpicture}}
\caption{ }\label{fig:comnetwork_b}
\end{subfigure}\\
\begin{subfigure}[b]{\textwidth}
\resizebox{1\textwidth}{!}{\begin{tikzpicture}[>=stealth',shorten >=1pt,node distance=3cm, on grid,initial/.style={}]
\tikzstyle{every state}=[fill=mycolor_lightblue]
  \node[state]          (4)                        {$\mathbf{4}$};
  \node[state]          (2) [above left =of 4]    {$\mathbf{2}$};
  \node[state]          (1) [left =of 2]    {$\mathbf{1}$};
  \node[state]          (3) [above right =of 4]    {$\mathbf{3}$};
 \node[state]          (5) [right=of 3]    {$\mathbf{5}$};
 %\node[state]          (T) [below=of 3] {text};
\tikzset{mystyle/.style={->,bend right=16, draw=mycolor_green, double=mycolor_green}}
\path (2)     edge [mystyle]    node [above]  {$2$} (1)
      (1)     edge [mystyle]    node [above]  {$3$} (2)
      (4)     edge [mystyle]    node [above]  {$2$} (3)
      (4)     edge [mystyle]    node [above]  {$2$} (5)
      (3)     edge [mystyle]    node [above]  {$2$} (4);
%  \path (2)     edge [mystyle]    node [above]  {$2$} (1)
%        (2)     edge [mystyle]    node [sloped, above]  {$2$} (3)
%        (4)     edge [mystyle]    node  [above] {$1$} (3)
%        (5)     edge [mystyle]    node  [above] {$1$} (4);
% \tikzset{mystyle/.style={->, bend right=16, double=red}}   
% \path (3)     edge [mystyle]    node  [sloped, above] {$3$} (4)
%       (1)     edge [mystyle]    node  [above] {$1$} (2)
%       (1)     edge [mystyle]    node  [sloped, below] {$2$} (3)    
%       (3)     edge [mystyle]    node  [sloped, above] {$3$} (2)     
%       (4)     edge [mystyle]    node  [above] {$4$} (5); 
% \tikzset{mystyle/.style={->, bend left=16, double=red}}       
% \path (5)     edge [mystyle]    node  [sloped, below] {$3$} (3);
\tikzset{mystyle/.style={->, draw=mycolor_green, double=mycolor_green}}
\path (1)     edge [mystyle]    node [sloped, above]  {$1$} (4)
      (2)     edge [mystyle]    node  [above ] {$5$} (4)
      (3)     edge [mystyle]    node  [above ] {$4$} (2)
      (5)     edge [mystyle]    node [above]  {$3$} (4);
 \end{tikzpicture}}
\caption{ }\label{fig:comnetwork_u}
\end{subfigure}
\end{subfigure}
 \begin{subfigure}[b]{0.32\textwidth}
                \psfrag{t}[cc][cc]{$t$} \psfrag{x}[cc][cc]{\small{$x_i$}}
                \psfrag{am}[cc][rr]{\scriptsize{$\mathsf{am}(\bm{x}(0))$}} \psfrag{gm}[cc][rr]{\scriptsize{$\mathsf{gm}(\bm{x}(0))$}}  
%\psfrag{0}[cc][cc]{\tiny$0$}\psfrag{1}[cc][cc]{\tiny$1$}
%\psfrag{2}[cc][cc]{\tiny$2$}\psfrag{3}[cc][cc]{\tiny$3$}
%\psfrag{4}[cc][cc]{\tiny$4$}\psfrag{5}[cc][cc]{\tiny$5$}
%\psfrag{6}[cc][cc]{\tiny$6$}\psfrag{7}[cc][cc]{\tiny$7$}
%\psfrag{10}[cc][cc]{\tiny$10$}
%\psfrag{-1}[cc][cc]{\tiny$-1$}
%\psfrag{-2}[cc][cc]{\tiny$-2$}\psfrag{-3}[cc][cc]{\tiny$-3$}
\includegraphics[scale=0.4]{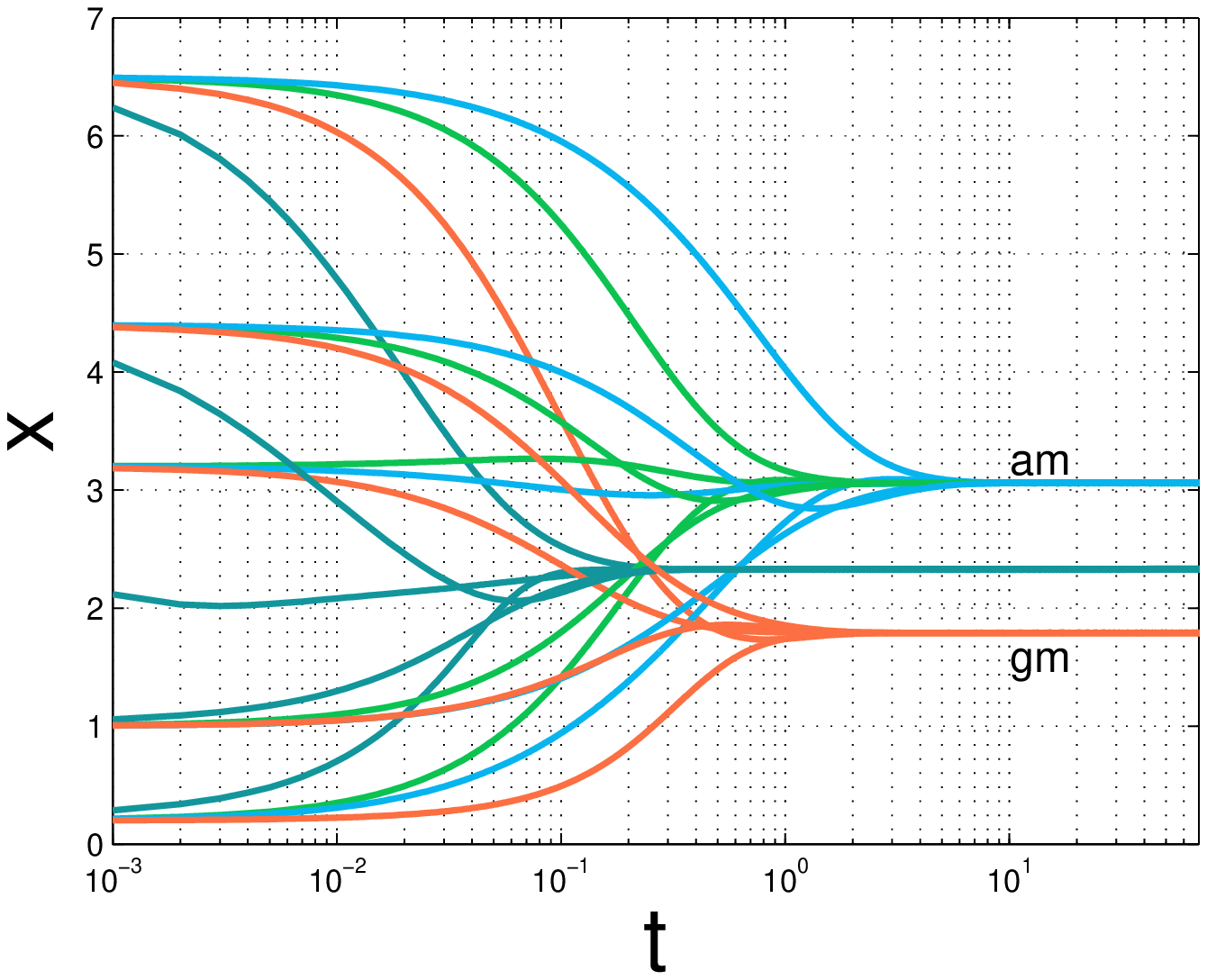}
 \caption{}\label{fig:tsts}
 \end{subfigure}
 \begin{subfigure}[b]{0.32\textwidth}       
  \psfrag{t}[cc][cc]{$t$} \psfrag{x}[cc][cc]{\small{$x_i$}}
  \psfrag{am}[cc][rr]{\scriptsize{$\mathsf{am}_w(\bm{x}(0))$}} \psfrag{gm}[cc][rr]{\scriptsize{$\mathsf{gm}_w(\bm{x}(0))$}} 
%\psfrag{0}[cc][cc]{\tiny$0$}\psfrag{1}[cc][cc]{\tiny$1$}
%\psfrag{2}[cc][cc]{\tiny$2$}\psfrag{3}[cc][cc]{\tiny$3$}
%\psfrag{4}[cc][cc]{\tiny$4$}\psfrag{5}[cc][cc]{\tiny$5$}
%\psfrag{6}[cc][cc]{\tiny$6$}\psfrag{7}[cc][cc]{\tiny$7$}
%\psfrag{10}[cc][cc]{\tiny$10$}
%\psfrag{-1}[cc][cc]{\tiny$-1$}
%\psfrag{-2}[cc][cc]{\tiny$-2$}\psfrag{-3}[cc][cc]{\tiny$-3$}
\includegraphics[scale=0.4]{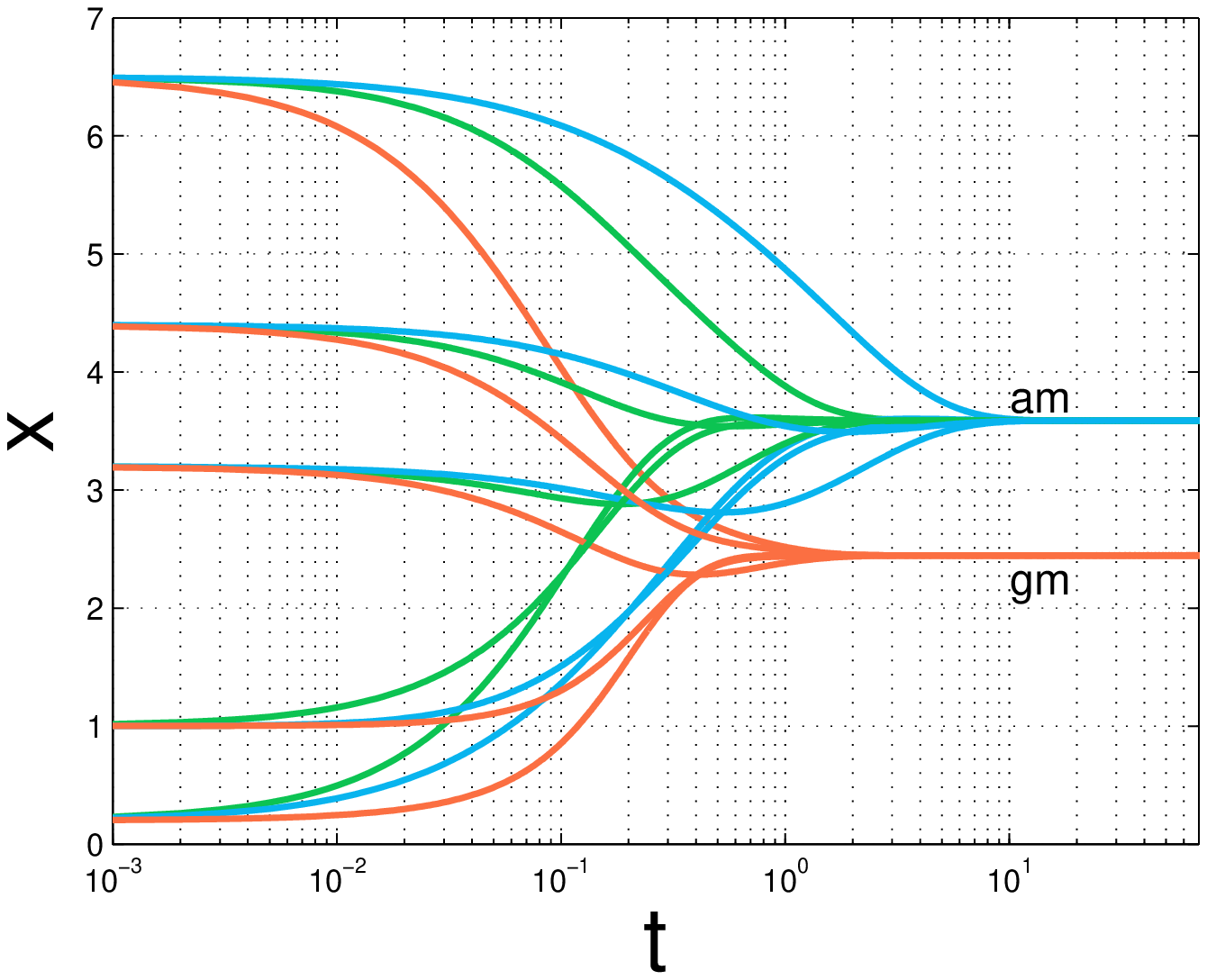} 
\caption{}\label{fig:tsts0k2}
 \end{subfigure}
 \caption{(a) depicts the underlying strongly connected balanced digraph, (b) the underlying strongly connected non-balanced graph, (c) and (d) show component trajectories for polynomial (teal), 
entropic (orange), scaling-invariant (blue), and standard linear consensus protocol (green) on graph  (a) and (b), respectively.}
\end{figure*}

\subsection{Numerical examples}
First, we compare the protocols of polynomial type \eqref{eq:1}, of entropic type \eqref{eq:2} and the scaling-invariant one 
\eqref{eq:3} for a digraph given in Fig.~\ref{fig:comnetwork_b}. This digraph is strongly connected and has balanced branch weights.  

For each of these protocols we compute trajectories starting at $\bm{x}(0)=[6.5,0.2,3.2,1,4.4]$. In accordance to Theorem \ref{thm:converg}, the novel network protocols are 
indeed consensus protocols that converge to a uniform equilibrium state $\bar{x}\bm{1}$. 
As the left-Perron vector for the balanced weighting is a uniform vector, the LTI consensus system must solve the average consensus problem with 
$\bar{x}=\mathsf{am}(\bm{x}(0))=3.06$,
the scaling-invariant protocol, according to Corollary~\ref{cor:wam}, as well, and
solution curves of the entropic protocol must converge to the geometric mean of the initial state, $\mathsf{gm}(\bm{x}(0))=1.7886=\bar{x}$, as shown in Theorem~\ref{thm:wgm}.
Our observations are confirmed by Fig.~\ref{fig:tsts}.

Next, let us illustrate the results in Theorem~\ref{thm:wgm} and Corollary~\ref{cor:wam} on a digraph which is strongly connected but not \mbox{balanced}. We consider a weighted digraph described in 
Fig.~\ref{fig:comnetwork_u}, which has Perron vector~\mbox{$\hat{\bm{q}}=[0.26,0.14,0.37,0.09,0.14]$}. The 
weighted arithmetic mean of the same initial condition using the Perron vector components as weights is~\mbox{$\mathsf{am}_w(\bm{x}(0))=3.5884$}, 
and the weighted geometric mean becomes~\mbox{$\mathsf{gm}_w(\bm{x}(0))=2.4444$}. 
% In Fig. \ref{fig:tsts0k2} component trajectories are plotted for 
% the linear, the scaling-invariant, and the entropic protocol in the same color-coding as before, using the same initial condition. 
Again, the simulation results as depicted in Fig. \ref{fig:tsts0k2} confirm our observations.

\section{Gradient and optimization viewpoint\label{sec:energetics}}

In this section we demonstrate that all three geometric mean driven consensus networks can be embraced in a common setting of a projected gradient flow of free energy. On that basis we provide a novel characterization of the geometric mean in terms of a constrained optimization problem.
\subsection{Free energy gradient flow}
Free energy stored in a state $\bm{x}$ w.r.t. another positive vector $\bm{y}$ can be defined as the
sum-separable function \cite{vdSchaft2013}
\begin{equation}\label{eq:fe}
F(\bm{x}||\bm{y}):=\sum_{i=1}^n x_i\left(\ln\frac{x_i}{y_i}-1\right) + const.
\end{equation}
For elements that are member of the set of vectors having total mass $m\in \mathbb{R}_{>0}$,
\begin{equation}
\mathscr{D}_m:=\left\{\bm{x}\in \mathbb{R}^n_{>0}:\sum_{i=1}^nx_i=m \right\},
\end{equation}
free energy is, up to an additive constant, a relative entropy; it coincides with the usual relative entropy known in information theory for vectors that are elements of the set of probability distribution vectors, $\mathscr{D}_{m=1}$, by setting $const=1$, so that
$F(\bm{x}||\bm{y})=\sum_{i=1}^n x_i\ln\frac{x_i}{y_i}$.
\begin{remark}
Within the literature on network systems relative entropy appears in the context of distributed estimation and detection algorithms, where the states represent discrete probabilities, see, e.g., \cite{ChungACC2014,QLiu2015, Zidek1986, NedicArxiv2015,JadbabaieCDC2013,JadbabaieScDir2015}.  Free energy is used in the study on mass-action chemical reaction networks \cite{vdSchaft2013}.
\end{remark}

In what follows we show that the polynomial, entropic, and scaling-invariant consensus dynamics are all instances of a particular type of a free energy gradient flow. 

To start with, an ODE governing a Riemannian gradient (descent) flow in $\mathbb{R}^n$ has the generic form
$\textbf{G}(\bm{x})\dot{\bm{x}}=-\nabla E(\bm{x})$
where $E:\mathbb{R}^n\to\mathbb{R}$ is the potential and $\textbf{G}:\mathbb{R}^n\to\mathbb{R}^{n\times n}$ is a positive definite matrix function smoothly varying in $\bm{x}$.
It defines the infinitesimal metric $\mathrm{d}\bm{x}\cdot\textbf{G}(\bm{x})\mathrm{d}\bm{x}$ in which a system is a gradient descent flow 
of $E$, so that $\textbf{G}^{-1}$  defines the inverse metric, cf., e.g., \cite{SimpsonPorco2014}.

Let $\textbf{L}$ be the symmetric Laplacian of an undirected connected graph. Using the eigen-decomposition
 $\textbf{L}=\textbf{V}\bm{\Lambda}\textbf{V}^\top$, where $\bm{\Lambda}:=\mathsf{diag}\{\lambda_1,\lambda_2,\ldots,\lambda_n\}$ is the diagonal matrix collecting eigenvalues of $\textbf{L}$, and $\textbf{V}$ collects orthogonal eigenvectors each having 2-norm one, we have
\begin{equation}\label{eq:projL}
\textbf{L}\bm{x}=\textbf{V}\bm{\Lambda}\textbf{V}^\top\bm{x}=\sum_{i=1}^{n}\bm{v}_i\, \lambda_i \bm{v}_i\cdot\bm{x},
\end{equation}
which is a projection of a vector $\bm{x}$ onto the set of distributions $\mathscr{D}_m, m=|\bm{x}|_1$.
To see this, recall that
a projection onto this set has the form
\begin{equation}\label{eq:proj}
\mathsf{Proj}_{\mathscr{D}} \bm{x}=\sum_{i=1}^{n-1} \frac{\bm{x}\cdot \tilde{\bm{v}}_i}{\tilde{\bm{v}}_i\cdot \tilde{\bm{v}}_i} \tilde{\bm{v}}_i,
\end{equation}
where $\{\tilde{\bm{v}}_1,\tilde{\bm{v}}_2,\ldots,\tilde{\bm{v}}_{n-1}\}$ are linearly independent vectors that span the hyperplane $\mathscr{D}_m$. This setting is given in \eqref{eq:projL}, as $\lambda_1=0$, while $\lambda_i>0, i=2,3,\ldots,n$, and $\bm{v}_1$ is orthogonal to any set $\mathscr{D}_m$. 
%In particular, for normalized symmetric weightings, the non-zero eigenvalues are all equal one.

Given a sum-separable convex function $\phi:\mathbb{R}^n_{>0}\to\mathbb{R}$ we introduce for the gradient $\nabla \phi$ projected onto $\mathscr{D}_m, m=|\nabla \phi|_1$, the notation $\nabla_{\mathscr{D}}\phi=\textbf{L}\nabla \phi$.
Observe that the gradient  $\nabla F(\bm{x}||\bm{1})$ is given by the vector $\ln\bm{x}$. 
Using the projected gradient notation, we can write
the protocols \eqref{eq:1}-\eqref{eq:3} in same order in vector matrix form as
\begin{align}
\dot{\bm{x}}&=-\textbf{R}(\bm{x})\textbf{L}\ln\bm{x} =-\textbf{R}(\bm{x})\nabla_{\mathscr{D}} F(\bm{x}||\bm{1}) \label{eq:1g}\\
\dot{\bm{x}}&=-\textbf{X}(\bm{x}) \textbf{L}\ln\bm{x}=-\textbf{X}(\bm{x})\nabla_{\mathscr{D}} F(\bm{x}||\bm{1}) \label{eq:2g}\\
\dot{\bm{x}}&=-\textbf{L}\ln\bm{x}=-\nabla_{\mathscr{D}} F(\bm{x}||\bm{1}),\label{eq:3g}
\end{align}
with $\textbf{L}$  the constant coefficient Laplacian, and $\textbf{R}(\bm{x}),\textbf{X}(\bm{x})$ as in the proof of Theorem \ref{thm:converg}.
As $\textbf{R}(\bm{x})$ and $\textbf{X}(\bm{x})$ are positive definite symmetric matrix functions for $\bm{x}\in\mathbb{R}_{>0}^{n}$, they 
define Riemannian metrics via their inverses. 

The scaling-invariant protocol on an undirected graph \mbox{generates} a projected gradient flow in the usual Euclidean metric setting. Therefore, according to the preceding discussion, on a completely normalized 
graph trajectories must evolve along steepest descent directions of free energy on the appropriate simplex of constant mass distributions.

In Fig. \ref{fig:fe3dtraj} this free energy gradient property is illustrated for the scaling-invariant protocol running on such a graph over three nodes. 
The gray outlined triangle marks the set of mass-3 distribution vectors.
Color-coded are iso-level curves of $F(\bm{x}||\bm{1})=\sum_i x_i(\ln x_i -1) + 3 $.
This illustration also highlights the appropriateness of the $n-1$-dimensional set of mass distribution vectors within positive $n$-space in considering the free energy functional:
Free energy is convex and permutation invariant on this set with minimum obtained at the consensus state.
 Three trajectories are plotted in black with initial conditions marked by a cross. We see that solution curves indeed follow steepest gradient descent directions of free energy on $\mathscr{D}_{m=3}$ being 
directed towards the minimum of this function, which is obtained at the consensus point.

\begin{figure}[t]
\centering
\psfrag{x}[l][cc]{\small$x_1$} \psfrag{y}[r][cc]{\small$x_2$}
\psfrag{z}[r][cc]{\small$x_3$}
 \psfrag{N}[cc][cc]{\small{number of simulation}}
\psfrag{0}[cc][cc]{\tiny$0$}\psfrag{1}[cc][cc]{\tiny$1$}
\psfrag{2}[cc][cc]{\tiny$2$}\psfrag{3}[cc][cc]{\tiny$3$}
\psfrag{0.5}[cc][cc]{\tiny$ $}\psfrag{1.5}[cc][cc]{\tiny$ $}
\psfrag{2.5}[cc][cc]{\tiny$ $}
\includegraphics[scale=0.5]{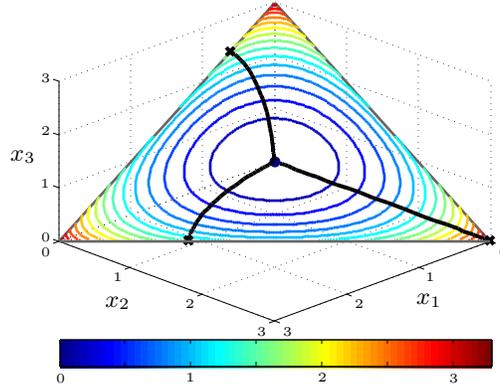}
\caption{Three trajectories generated by scaling-invariant protocol 
converging to consensus in a free energy potential on the simplex $\mathscr{D}_{m=3}$}
\label{fig:fe3dtraj}
\end{figure}

\subsection{Constrained non-linear optimization view}
%The fact that the entropic protocol on a balanced graph converges to the geometric mean of the initial state 
%components together with the particular free energy gradient descent property yields a novel variational characterization of the geometric mean that links 
%dynamic problems in consensus theory with static problems in non-linear constrained optimization.
Motivated by the preceding results for the entropic protocol we provide a novel variational characterization of the \mbox{geometric} mean linking dynamic problems in consensus theory with static problems in non-linear constrained optimization.

\begin{theorem}[Novel characterization of the geometric mean]
%Let $\mathbb{R}_{>e^{-1}}^n$ be the set of real $n$-vectors having component values greater than $\frac{1}{e}$.
The geometric mean of a vector $\bm{x}\in \mathbb{R}_{>0}^n$ is characterized as the value $\mathsf{am}(\bm{x}^*)$, where 
\begin{equation}
\bm{x}^*= \argmin_{\bm{y}\in \mathbb{R}_{>0}^n} F(\bm{y}||\bm{1}), \ \ \mathrm{subject\; to}\; \prod_{i=1}^ny_i=\prod_{i=1}^n x_i.\label{eq:minF}
\end{equation}
That is, $\bm{x}^*$ minimizes free energy on the manifold of states having constant product of component values. In particular, this vector has the form of a consensus state with 
agreement value precisely the geometric mean of $\bm{x}$.
\end{theorem}

\begin{proof}
Define the Lagrangian
\begin{equation}
\mathcal{L}(\bm{y},\lambda)=F(\bm{y}||\bm{1})-\lambda \left(\prod_iy_i -\prod_i x_i\right).
\end{equation}
The solution of the constrained free energy minimization problem satisfies
the first order optimality conditions
\begin{align}
\nabla_{\lambda} \mathcal{L} &= \prod_{j=1}^n x_j- \prod_{j=1}^n y_j=0 \Leftrightarrow
\prod_{j\not = i}y_j = \frac{\prod_{k=1}^nx_k}{y_i}, \label{eq:Laglambda} \\
\nabla_{y_i} \mathcal{L} &=\ln y_i -\lambda \prod_{j\not = i}y_j=0,
 \ \ i \in N \label{eq:LagX}
\end{align}
which consequently leads to the solution characteristic
\begin{equation}
y_i\ln y_i=\lambda \prod_{k=1}^nx_k= constant, \ \ \forall i\in N. \label{eq:solchar}
\end{equation}
The right hand side in \eqref{eq:solchar} is positive (the multiplier $\lambda$ is positive and the values $x_i>0$ by assumption) and the function $y \ln y$ is increasing on the domain where it takes positive values. Therefore, \eqref{eq:solchar} has a unique solution, and this solution
is the same for all $i\in N$, i.e., a consensus state.

Next, we show that the agreement value of the consensus state is the geometric mean of $\bm{x}$. 
Writing $\bm{y}=y\bm{1}$ and substituting into \eqref{eq:Laglambda} yields
\begin{equation}
\prod_{k=1}^n y_k=y^{n}=\prod_{k=1}^n x_k \Leftrightarrow y=\mathsf{gm}(\bm{x}).
\end{equation}
That is, if $\bm{x}\in \mathbb{R}_{>0}^n$, then the solution of \eqref{eq:minF} is
$\bm{x}^*=\mathsf{gm}(\bm{x})\bm{1}$, so that $\mathsf{am}(\bm{x}^*)=\mathsf{gm}(\bm{x})$.
\end{proof}
%While the usual variational characterization of the geometric mean has the form of an unconstrained optimization problem involving a cost function that is not sum-separable and of the metric 
%interaction type, the characterization in terms of a constrained minimization problem involves a sum-separable cost function. 
Sum-separable energy functions play an axiomatic role in dissipative 
interconnected systems \cite{willems1972a} where they represent energy stored in local subsystems. In contrast, energy functions of the interaction type usually represent power dissipated 
``across'', e.g., resistor elements, see for instance \cite{Mangesius2016} for a further discussion.
Hence, the free energy minimization property seems to be the natural gradient setting for the time-continuous entropic consensus network when seen as analog circuit device solving a minimization 
problem.

\section{Conclusion}
In this paper we propose and study novel non-linear continuous-time consensus protocols driven in three distinct ways by the geometric mean: the polynomial, the entropic, and the scaling-invariant consensus protocols. 
The three protocols are aligned in a free energy gradient property on the simplex of constant mass distribution vectors.
%We exposed a relationship between our polynomial consensus and the polynomial describing reaction rates in chemical kinetics, 
%as well as a relationship between the solution of an elliptic integral and a lower bound on the achieved consensus value obtained by the protocol running on normalized balanced graphs.
The entropic consensus dynamics represents a generalization of the well-known \mbox{average} consensus problem as the asymptotically reached agreement value corresponds to the (weighted) geometric mean of the initial state. 
Based on the free energy gradient property for the entropic dynamics, we provide a novel variational \mbox{characterization} of the geometric mean using a non-linear constrained optimization problem.

%\section*{Acknowledgement}
%The work is partially supported by the German Research Foundation (DFG) within the Priority Program SPP 1914 “Cyber-Physical Networking”, the EU H2020 Innovative Training Network (ITN) "Platform-aware Model-driven Optimization of Cyber-Physical Systems (oCPS)" under grant agreement no. 674875, and the TUM Institute for Advanced Study.

\begin{appendix}

\section{Agreement value polynomial case --- numerical study} \label{sec:numstudy}
In the following we study the consensus value of the polynomial protocol on a normalized balanced digraph using numerical simulations. We observe that the consensus value can be upper bounded by the arithmetic mean of the initial state, and lower bounded by the arithmetic-geometric mean of the arithmetic mean and the geometric mean of the initial condition.

The arithmetic-geometric mean $\mathsf{agm}(a,b)$ of two positive numbers $a,b$, can be defined as the limiting point of a discrete time 
dynamical system, $\{a_k,b_k\}_{k\geq 0}$, $k\in \mathbb{N}$ that satisfies the algorithmic update rule
\begin{equation}\label{eq:agmalgorithm}
\begin{pmatrix}
a_{k+1} \\
b_{k+1}
\end{pmatrix}
= \begin{pmatrix}
\mathsf{am}(\{a_k,b_k\})\\
\mathsf{gm}(\{a_k,b_k\})
\end{pmatrix}.
\end{equation}
It is obtained as the limit
\begin{equation}
\mathsf{agm}(a,b):=\lim_{k\to\infty} a_k =\lim_{k\to\infty} b_k, \ \ a_0=a, b_0=b;
\end{equation}
The fixed-point iteration 
\eqref{eq:agmalgorithm} is due to Carl Friedrich Gauss, who was concerned with computing the perimeter of ellipses, which up until today is a topic of scientific discourse \cite{Adlaj2012}\cite{Borwein1987}.
The arithmetic-geometric mean is related to the solution of a complete elliptic integral, as
\begin{equation}
\mathsf{agm}(a,b)=\frac{\pi}{2}\frac{1}{I(a,b)},\, I(a,b):=\int_0^{\frac{\pi}{2}}\frac{\mathrm{d}\varphi}{\sqrt{a^2\cos^2\varphi+b^2\sin^2\varphi}},
\end{equation}
see, e.g., \cite{Carlson1971}.

We first consider completely connected normalized balanced graphs, that differ only in the number of nodes, such that $N\in \{2,3,\ldots,50\}$.
For each of these graphs we run the polynomial protocol for 50 random initial conditions sampled from the interval $]0,10[$, such that $\mathsf{am}(\bm{x}(0))=c_1$, and 
$\mathsf{gm}(\bm{x}(0))=c_2$, 
where $c_1>c_2>0$. 
In Fig. \ref{fig:ConsValcomplete} the reached agreement values for this experiment 
are plotted as black circles. The red squares show the arithmetic mean value of the initial condition, sampled such that $c_1=4$, and the blue squares represent the geometric mean of the initial states, sampled such that $c_2=3$.
We observe that for each graph, every of the reached consensus values lies above the green line,
which appears to be a tight and strict lower bound.
We found that the value of the green marks computes as the arithmetic-geometric mean of the arithmetic and the geometric mean of the initial state, i.e., its value corresponds to the number $\mathsf{agm}(c_1,c_2)$.

\begin{figure}[h]
\centering
\psfrag{y}[cc][cc]{$\bar{x}$} \psfrag{N}[cc][cc]{Number of nodes}
\psfrag{0}[cc][cc]{\tiny$0$}\psfrag{10}[cc][cc]{\tiny$10$}
\psfrag{20}[cc][cc]{\tiny$20$}\psfrag{30}[cc][cc]{\tiny$30$}
\psfrag{40}[cc][cc]{\tiny$40$}\psfrag{50}[cc][cc]{\tiny$50$}
\psfrag{4}[cc][cc]{\tiny$4$}\psfrag{3}[cc][cc]{\tiny$3$}
\psfrag{2.8}[cc][cc]{\tiny$2.8 $}\psfrag{3.2}[cc][cc]{\tiny$3.2 $}
\psfrag{3.4}[cc][cc]{\tiny$3.4 $}\psfrag{3.8}[cc][cc]{\tiny$3.8 $}
\psfrag{4.2}[cc][cc]{\tiny$4.2$}
%\psfrag{0.5}[cc][cc]{$ $}\psfrag{1.5}[cc][cc]{$ $}
%\psfrag{2.5[cc][cc]{$ $}
\includegraphics[scale=0.5]{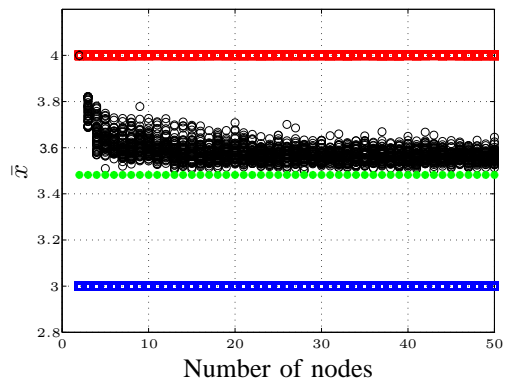}
\caption{Consensus values (black) for all-to-all normalized balanced graphs for 50 random initial conditions such that the arithmetic mean of initial condition (red square) takes value 4 and the 
geometric mean (blue square) has value 3. The green marks represent $\mathsf{agm}(3,4)$.}
\label{fig:ConsValcomplete}
\end{figure}

To verify that this observation is independent of the set mean value constraints $c_1$, $c_2$, we next consider the polynomial protocol on a completely connected, balanced, normalized graph with number of nodes being fixed at $N=5$. We are interested in the values of the ratio
$\frac{\mathsf{ref}}{\bar{x}}$, where $\mathsf{ref}\in \{\mathsf{am}(\bm{x}(0)),\mathsf{gm}(\bm{x}(0)),\mathsf{agm}(\mathsf{am}(\bm{x}(0)),\mathsf{gm}(\bm{x}(0))) \}$.
Clearly, the closer this fraction is to one, the better is ``$\mathsf{ref}$'' suited as an estimate for the asymptotically reached consensus value, given on the basis of the initial data.

In Fig. \ref{fig:consValagm} we plotted this ratio $\frac{\mathsf{ref}}{\bar{x}}$ for
500 random initializations sampled from the interval $]0,10[$. The red dots mark this ratio for $\mathsf{ref}=\mathsf{am}(\bm{x}(0))$, the blue ones for $\mathsf{ref}=\mathsf{gm}(\bm{x}(0))$, and the green ones mark the ratio for reference taken as arithmetic-geometric mean of the arithmetic mean and the geometric mean of the initial state.
We can confirm the previous observation that for each trajectory the arithmetic mean of the initial condition
is an upper bound for the consensus value (red dots mark above one), the geometric mean a lower bound (green dots mark below one), and so is the arithmetic-geometric mean (green dots mark below one), whereas this value is a tighter lower bound than the geometric mean. In particular, the arithmetic-geometric mean bound appears to be in many cases a good estimate of the achieved consensus value as the green dots cluster very near to the black line.
\begin{figure}[h]
\centering
\hspace*{-0.6cm}
\psfrag{y}[cc][cc]{$ $} \psfrag{N}[cc][cc]{\small{Number of simulation}}
\psfrag{0}[cc][cc]{\tiny$0$}\psfrag{1000}[cc][cc]{\tiny$1$}
\psfrag{2000}[cc][cc]{\tiny$2$}\psfrag{3000}[cc][cc]{\tiny$3$}
\psfrag{4000}[cc][cc]{\tiny$4$}\psfrag{5000}[cc][cc]{\tiny$5$}
\psfrag{6000}[cc][cc]{\tiny$7$}\psfrag{7000}[cc][cc]{\tiny$7$}
\psfrag{8000}[cc][cc]{\tiny$8$}\psfrag{9000}[cc][cc]{\tiny$9$}
\psfrag{10000}[l][l]{\tiny$10 \times 10^3$}
\psfrag{0.2}[r][r]{\tiny$0.2$}\psfrag{0.4}[r][r]{\tiny$0.4$}
\psfrag{0.6}[r][r]{\tiny$0.6$}\psfrag{0.8}[r][r]{\tiny$0.8$}
\psfrag{1}[r][r]{\tiny$1$} \psfrag{1.2}[r][r]{\tiny$1.2$}\psfrag{1.4}[r][r]{\tiny$1.4$} \psfrag{1.6}[r][r]{\tiny$1.6$}\psfrag{1.8}[r][r]{\tiny$1.8$}
%\psfrag{0.5}[cc][cc]{$ $}\psfrag{1.5}[cc][cc]{$ $}
%\psfrag{2.5[cc][cc]{$ $}
\includegraphics[scale=0.5]{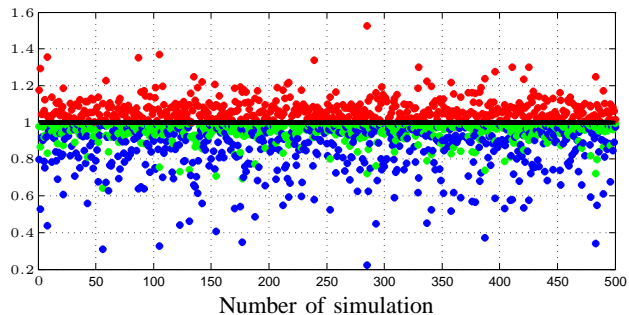}%consensusValAGM10150
\caption{Consensus ratios $\mathsf{ref}/\bar{x}$, $\mathsf{ref}=\mathsf{am}(\bm{x}(0))$ (red), $\mathsf{ref}=\mathsf{gm}(\bm{x}(0))$ (blue), $\mathsf{ref}=\mathsf{agm}\{\mathsf{am}(\bm{x}(0)),\mathsf{gm}(\bm{x}(0)\}$ (green), and $\mathsf{ref}=\bar{x}$ (black) for $500$ simulations of a normalized complete balanced graph on $5$ nodes with initial conditions randomly sampled from the interval $]0,10[$.
}
\label{fig:consValagm}
\end{figure}

Eventually, we test if the $\mathsf{agm}$ as lower bound is independent of the normalization of the weighting and independent of the number of connected nodes
that is, if it is a lower bound for the consensus value for every $(N,d)$-regular graph, i.e., balanced graphs on $N$ nodes with $d\in 
\mathbb{N}$ nodes being connected to each node $i\in N$.
In Fig. \ref{fig:agmpermean} the ratio $\mathsf{agm}(c_1,c_2)/\bar{x}$, $c_1=\mathsf{am}(\bm{x}_0)$, $c_2=\mathsf{gm}(\bm{x}_0)$  is plotted for $N=30$, $d\in\{2,3,\ldots,22\}$, where the red dots 
mark the defined ratio for non-normalized unweighted balanced graphs, (i.e., $w_{ij}\in \{0,1\}$), and
the blue dots mark this ratio for normalized ones. For each graph we computed $30$ trajectories for random initial conditions sampled as before. We see that the $\mathsf{agm}$ lower bound holds only for the normalized case; it is independent of the degree $d$.

\begin{figure}[h]
\centering
\psfrag{y}[cc][cc]{$\mathsf{agm}/\bar{x} $} \psfrag{deg}[cc][cc]{\small{degree}}
\psfrag{0}[cc][cc]{\tiny$0$}\psfrag{1000}[cc][cc]{\tiny$1$}
\psfrag{2000}[cc][cc]{\tiny$2$}\psfrag{3000}[cc][cc]{\tiny$3$}
\psfrag{4000}[cc][cc]{\tiny$4$}\psfrag{5000}[cc][cc]{\tiny$5$}
\psfrag{6000}[cc][cc]{\tiny$7$}\psfrag{7000}[cc][cc]{\tiny$7$}
\psfrag{8000}[cc][cc]{\tiny$8$}\psfrag{9000}[cc][cc]{\tiny$9$}
\psfrag{10000}[l][l]{\tiny$10 \times 10^3$}
\psfrag{0.2}[r][r]{\tiny$0.2$}\psfrag{0.4}[r][r]{\tiny$0.4$}
\psfrag{0.6}[r][r]{\tiny$0.6$}\psfrag{0.8}[r][r]{\tiny$0.8$}
\psfrag{1}[r][r]{\tiny$1$} \psfrag{1.2}[r][r]{\tiny$1.2$}\psfrag{1.4}[r][r]{\tiny$1.4$} \psfrag{1.6}[r][r]{\tiny$1.6$}\psfrag{1.8}[r][r]{\tiny$1.8$}
%\psfrag{0.5}[cc][cc]{$ $}\psfrag{1.5}[cc][cc]{$ $}
%\psfrag{2.5[cc][cc]{$ $}
\includegraphics[scale=0.5]{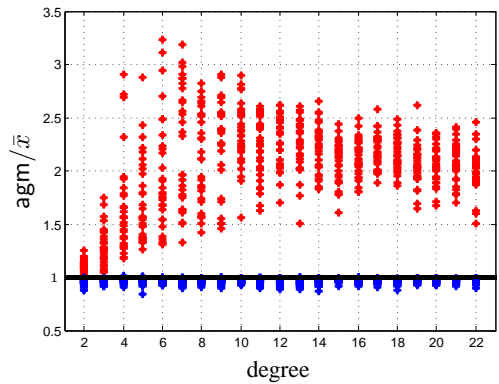}%consensusValAGM10150
\caption{Ratio $\mathsf{agm}(c_1,c_2)/\bar{x}$, $c_1=\mathsf{am}(\bm{x}_0)$, $c_2=\mathsf{gm}(\bm{x}_0)$ for $(N,d)$-regular graphs, $N=30$, $d\in\{2,3,\ldots,22\}$; 
non-normalized weighting (red) and normalized weighting (blue).
}
\label{fig:agmpermean}
\end{figure}

\end{appendix}

\bibliographystyle{ieeetr}
{\small
\bibliography{ConsGeoMeanBib}

\begin{thebibliography}{10}

\bibitem{Olfati-Saber2007}
R.~Olfati-Saber, J.~Fax, and R.~Murray, ``Consensus and cooperation in
  networked multi-agent systems,'' {\em Proceedings of the IEEE}, vol.~95,
  pp.~215--233, Jan 2007.

\bibitem{Tsitsiklis1989}
D.~P. Bertsekas and J.~Tsitsiklis, {\em Parallel and Distributed Computation:
  Numerical Methods}.
\newblock Prentice-Hall, 1989.

\bibitem{Moreau2004}
L.~Moreau, ``Stability of continuous-time distributed consensus algorithms,''
  in {\em 43rd IEEE Conference on Decision and Control}, pp.~3998--4003, 2004.

\bibitem{Acebron2005}
J.~A. Acebr\'{o}n, L.~L. Bonilla, C.~J.~P. Vicente, F.~Ritort, and R.~Spigler,
  ``The {K}uramoto model: A simple paradigm for synchronization phenomena,''
  {\em Reviews of Modern Physics}, vol.~77, no.~1, pp.~137--185, 2005.

\bibitem{Doerfler2014}
F.~D\"orfler and F.~Bullo, ``Synchronization in complex networks of phase
  oscillators: A survey,'' {\em Automatica}, vol.~50, pp.~1539--1564, 2014.

\bibitem{Hendrickx2014}
F.~D\"{o}rfler and J.~M. Hendrickx, ``Synchronization of oscillators:
  Feasibility and non-local analysis,'' in {\em LCCC Workshop on Dynamics and
  Control in Networks
  \url{http://www.lccc.lth.se/index.php?page=open-problems}}, 2014.

\bibitem{Breakspear2010}
M.~Breakspear, S.~Heitmann, and A.~Daffertshofer, ``Generative models of
  cortical oscillations: from {K}uramoto to the nonlinear {F}okker {P}lanck
  equation,'' {\em Frontiers in Human Neuroscience}, vol.~4, no.~190,
  p.~10.3389/fnhum.2010.00190, 2010.

\bibitem{Sepulchre2011}
R.~Sepulchre, ``Consensus on nonlinear space,'' {\em Annual reviews in
  control}, vol.~35, no.~1, pp.~56--64, 2011.

\bibitem{SarletteSIAM2009}
A.~Sarlette and R.~Sepulchre, ``Consensus optimization on manifolds,'' {\em
  SIAM Journal on Control and Optimization}, vol.~48, pp.~56--76, 2009.

\bibitem{Scardovi2007}
L.~Scardovi, A.~Sarlette, and R.~Sepulchre, ``Synchronization and balancing on
  the n-torus,'' {\em Systems \& Control Letters}, vol.~56, no.~5,
  pp.~335--341, 2007.

\bibitem{Jadbabaie2004}
A.~Jadbabaie, N.~Motee, and M.~Barahona, ``On the stability of the {K}uramoto
  model of coupled nonlinear oscillators,'' in {\em American Control
  Conference, 2004. Proceedings of the 2004}, vol.~5, pp.~4296--4301 vol.5,
  June 2004.

\bibitem{Spizman2008}
L.~Spizman and M.~A. Weinstein, ``A note on utilizing the geometric mean: when,
  why and how the forensic economist should employ the geometric mean,'' {\em
  Journal of Legal Economics}, vol.~15, pp.~43--55, 2008.

\bibitem{Zenner2008}
M.~Zenner, S.~Hill, J.~Clark, and N.~Mago, ``The most important number in
  finance - the quest for the market risk premium,'' tech. rep., JP Morgan,
  2008.

\bibitem{Mitchell2004}
D.~W. Mitchell, ``More on spreads and non-arithmetic means,'' {\em The
  Mathematical Gazette}, vol.~88, no.~511, pp.~142--144, 2004.

\bibitem{Shingleton2010}
A.~Shingleton, ``Allometry: The study of biological scaling,'' {\em Nature
  Education Knowledge}, vol.~3, no.~10, p.~2, 2010.

\bibitem{Connors1990}
K.~A. Connors, {\em Chemical Kinetics - The Study of Reaction Rates in
  Solution}.
\newblock VCH Publishers, Inc., 1990.

\bibitem{HaraAutomatica2011}
Y.~Hori, T.-H. Kim, and S.~Hara, ``Existence criteria of periodic oscillations
  in cyclic genere regulatory networks,'' {\em Automatica}, vol.~47,
  pp.~1203--1209, 2011.

\bibitem{ChungACC2014}
S.~Bandyopadhyay and S.-J. Chung, ``Distributed estimation using {B}ayesian
  consensus filtering,'' in {\em IEEE American Control Conference (ACC)},
  pp.~634--641, 2014.

\bibitem{QLiu2015}
L.~Qipeng, Z.~Jiuhua, and W.~Xiaofan, ``Distributed detection via {B}ayesian
  updates and consensus,'' in {\em 34th Chinese Control Conference (CCC)},
  pp.~6992--6997, 2015.

\bibitem{Zidek1986}
C.~Genest and J.~V. Zidek, ``Combining probability distributions: A critique
  and an annotated bibliography,'' {\em Statistical Science}, vol.~1, no.~1,
  pp.~114--148, 1986.

\bibitem{NedicArxiv2015}
A.~Nedi\'{c}, A.~Olshevsky, and C.~A. Uribe, ``Fast convergence rates for
  distributed non-{B}ayesian learning,'' {\em arXiv:1508.05161v1 [math.OC]},
  2015.

\bibitem{JadbabaieCDC2013}
S.~Shahrampour and A.~Jadbabaie, ``Exponentially fast parameter estimation in
  networks using distributed dual averaging,'' in {\em 52nd IEEE Conference on
  Decision and Control}, pp.~6196--6201, 2013.

\bibitem{JadbabaieScDir2015}
M.~A. Rahimian and A.~Jadbabaie, ``Learning without recall: A case of
  log-linear learning,'' {\em IFAC-PapersOnLine}, vol.~48, no.~22, pp.~46--51,
  2015.

\bibitem{Krause2005}
R.~Hegselmann and U.~Krause, ``Opinion dynamics driven by various ways of
  averaging,'' {\em Computational Economics}, vol.~25, pp.~381--405, 2005.

\bibitem{Baras2015}
I.~Matai and J.~S. Baras, ``The asymptotic consensus problem on convex metric
  spaces,'' {\em IEEE Transactions on Automatic Control}, vol.~60, no.~4,
  pp.~907 -- 921, 2015.

\bibitem{Bishop2014}
A.~Bishop and A.~Doucet, ``Distributed nonlinear consensus in the space of
  probability measures,'' in {\em Proceedings of the 19th IFAC World Congress},
  2014.

\bibitem{Stahl1993}
S.~Stahl, {\em The Poincar\'{e} Half-Plane - A Gateway to Modern Geometry}.
\newblock Jones and Bartlett Publishers, 1993.

\bibitem{Hendrickx2013}
J.~M. Hendrickx and J.~N. Tsitsiklis, ``Convergence of type-symmetric and
  cut-balanced consensus seeking systems,'' {\em IEEE Transaction on Automatic
  Control}, vol.~58, pp.~214--218, 2013.

\bibitem{Murray2004}
R.~Olfati-Saber and R.~M. Murray, ``Consensus problems in networks of agents
  with switching topology and time-delays,'' {\em IEEE Transactions on
  Automatic Control}, vol.~49, p.~9, 2004.

\bibitem{vdSchaft2013}
A.~van~der Schaft, S.~Rao, and B.~Jayawardhana, ``On the mathematical structure
  of balanced chemical reaction networks governed by mass action kinetics,''
  {\em SIAM Journal on Applied Mathematics}, vol.~73, no.~2, pp.~953--973,
  2013.

\bibitem{vdSchaftIFAC2013}
A.~van~der Schaft, S.~Rao, and B.~Jayawardhana, ``On the network thermodynamics
  of mass action chemical reaction networks,'' in {\em 1st IFAC Workshop on
  Thermodynamic Foundations of Mathematical Systems Theory}, pp.~24--29, 2013.

\bibitem{Yong2012}
W.-A. Yong, ``Conservation-dissipation structure of chemical reaction
  systems,'' {\em Physical Review E}, vol.~86, p.~067101, 2012.

\bibitem{Cover1991}
T.~M. Cover and J.~A. Thomas, {\em Elements of Information Theory}.
\newblock John Wiley \& Sons, Inc., 1991.

\bibitem{Haddad2008}
Q.~Hui and W.~M. Haddad, ``Distributed nonlinear control algorithms for network
  consensus,'' {\em Automatica}, vol.~44, pp.~2375--2381, 2008.

\bibitem{MurrayACC2003}
R.~Olfati-Saber and R.~M. Murray, ``Consensus protocols for networks of dynamic
  agents,'' in {\em American Control Conference}, pp.~951--956, 2003.

\bibitem{Sepulchre2010}
R.~Sepulchre, A.~Sarlette, and P.~Rouchon, ``Consensus in non-commutative
  spaces,'' in {\em Decision and Control (CDC), 2010 49th IEEE Conference on},
  pp.~6596--6601, Dec 2010.

\bibitem{Moreau2005}
L.~Moreau, ``Stability of multiagent systems with time-dependent communication
  links,'' {\em IEEE Transactions on Automatic Control}, vol.~50, pp.~169--182,
  2005.

\bibitem{SimpsonPorco2014}
J.~W. Simpson-Porco and F.~Bullo, ``Contraction theory on
  \textnormal{Riemannian} manifolds,'' {\em Systems $\&$ Control Letters},
  vol.~65, pp.~74--80, 2014.

\bibitem{willems1972a}
J.~C. Willems, ``Dissipative dynamical systems part 1: General theory,'' {\em
  Archive for Rational Mechanics and Analysis}, vol.~45, pp.~321-- 351, 1972.

\bibitem{Mangesius2016}
H.~Mangesius, J.-C. Delvenne, and S.~K. Mitter, ``Gradient and passive circuit
  structure in a class of non-linear dynamics on a graph,'' {\em Systems and
  Control Letters}, vol.~96, pp.~30-- 36, 2016.

\bibitem{Adlaj2012}
S.~Adlaj, ``An eloquent formula for the perimeter of an ellipse,'' {\em Notices
  of the AMS}, vol.~59, pp.~1094--1099, 2012.

\bibitem{Borwein1987}
J.~M. Borwein and P.~B. Borwein, {\em Pi and the AGM - A Study in Analytic
  Number Theory and Computational Complexity}.
\newblock John Wiley \& Sons, Inc., 1987.

\bibitem{Carlson1971}
B.~Carlson, ``Algorithm involving arithmetic and geometric means,'' {\em The
  American Mathematical Monthly}, vol.~78, pp.~496--505, 1971.

\end{thebibliography}
}

%\begin{IEEEbiography}[{\includegraphics[width=1in,height=1.25in,clip,keepaspectratio]{head_hirche}}]{text}

%

%\begin{IEEEbiographynophoto}{Sandra Hirche}
%% insert where needed to balance the two columns on the last page with
%%\newpage
%%
%\begin{IEEEbiographynophoto}{xx}


%\vfill

\end{document}